\documentclass[reqno]{amsart}

\usepackage{amsmath,amssymb,mathtools,amsopn,amsfonts,amsthm}
\usepackage{subfigure}
\usepackage{a4wide}
\usepackage[utf8]{inputenc}
\usepackage{enumerate, enumitem}
\usepackage[UKenglish]{babel}
\usepackage{verbatim}
\usepackage[initials]{amsrefs}
\usepackage{hyperref}

\usepackage{pgf,tikz,pgfplots}
\usepackage{mathrsfs}
\usepackage{graphicx}
\graphicspath{{tikz pics/}}
\usepackage{grffile}
\usetikzlibrary{arrows.meta, calc,patterns,angles,quotes}
\usetikzlibrary{decorations.pathreplacing,decorations.markings}
\tikzstyle arrowstyle=[scale=1.5]
\tikzstyle directed=[postaction={decorate,decoration={markings,
  mark=at position 0.48 with {\arrow[arrowstyle]{{To}}}}}]

\newtheorem{theorem}{Theorem}[section]
\newtheorem{lemma}[theorem]{Lemma}
\newtheorem{corollary}[theorem]{Corollary}
\newtheorem{proposition}[theorem]{Proposition}
\newtheorem*{theorem*}{Theorem}
\theoremstyle{definition}
\newtheorem{definition}[theorem]{Definition}

\newtheorem{remark}[theorem]{Remark}

\numberwithin{equation}{section}

\newcommand{\PSL}{\mathrm{PSL}(2,\mathbb{R})}
\newcommand{\SL}{\mathrm{SL}(2,\mathbb{R})}

\newcommand{\R}{\mathbb{R}}
\newcommand{\tr}{\textnormal{tr}}
\newcommand{\D}{\mathbb{D}}

\newcommand{\N}{\mathbb{N}}

\newcommand{\C}{\mathbb{C}}

\renewcommand{\H}{\mathbb{H}}
\renewcommand{\P}{\mathbb{P}}

\newcommand{\A}{\mathcal{A}}

\renewcommand{\geq}{\geqslant}
\renewcommand{\leq}{\leqslant}
\renewcommand{\subset}{\subseteq}
\renewcommand{\epsilon}{\varepsilon}

\setlist[enumerate]{leftmargin=20pt,itemsep=0pt,topsep=0pt}
\setlist[enumerate,1]{label=\emph{(\roman*)},ref={(\roman*)}}

\let\Im\relax
\DeclareMathOperator{\Im}{Im}
\let\Re\relax
\DeclareMathOperator{\Re}{Re}

\title[Geometric conditions for matrix domination]{Geometric conditions for matrix domination in two dimensions}

\author{Argyrios Christodoulou}
\address{Department of Mathematics, Aristotle University of Thessaloniki, Greece, 54124.}
\email{argyriac@math.auth.gr}
\subjclass[2020]{Primary: 37D20, 37D30 15A60; Secondary: 30F45}
\keywords{Dominated splitting; products of matrices; hyperbolic geometry}

\begin{document}

\maketitle

\begin{abstract}
In this article we prove a necessary and a sufficient condition for a finite subset of the special linear group to be dominated. These conditions are purely geometric in nature, as they only involve the trace and the eigenvectors of the matrices, and can be computed explicitly. Our sufficient condition, in particular, provides a simple algorithm for constructing a dominated set with prescribed eigenvectors, which is also applicable to certain Fuchsian groups. Our techniques involve an intuitive way of quantifying the geometric configuration of pairs of matrices, and take advantage of the interaction between dominated sets and two-dimensional hyperbolic geometry.
\end{abstract}

\section{Introduction}

The notion of a dominated splitting for a dynamical system was introduced by Ma\~ne \cite{Ma1978, Ma1984} in his seminal work on the $C^1$-stability Conjecture, as a weaker version of the notion of hyperbolicity introduced by Smale. It has since appeared in various areas of the theory of dynamical systems, such as the continuity of Lyapunov exponents \cite{BoVi2005}, differentiable dynamics \cite{BoDiPu2003}, Anosov Representations \cite{BoPoSa2019}, control theory \cite{CoKl2000} and the continuity properties of the lower spectral radius \cite{BoMo2015}. A detailed overview of dominated splittings, often simply called \emph{domination}, can be found in the survey \cite{Sa2016}.

The main focus of this article is dominated subsets of $\SL$. The two-dimensional case is particularly interesting, since domination is linked with dimension theory \cite{ChJu2021, SoTa2021} and the spectral theory of Schroedinger operators \cite{DaFiGo2022, Wo2022}; applications unique to the setting of $\SL$.

Let us start by providing the definition of domination in this setting. Throughout the article, we equip $\SL$ with an operator norm and with the induced topology.

\begin{definition}\label{def: domination (norm)}
A finite set $\A=\{A_1,A_2,\dots,A_N\}\subset\SL$ will be called \emph{dominated} if there exist constants $C>0$ and $\lambda>1$, so that 
\[
\lVert A_{i_n}A_{i_{n-1}}\cdots A_{i_1}\rVert \geq C \lambda^n, \quad \text{for any sequence}\ (i_j)\subset\{1,2,\dots,N\}\ \text{and all}\ n\in\N.
\]
\end{definition}

This particular characterization of domination---that is best suited for our purposes---was obtained by Yoccoz in \cite{Yo2004} in his study of the hyperbolicity properties of linear cocycles. Also, in this two-dimensional setting domination is often referred to as \emph{uniform hyperbolicity} (see, for example, \cite{AvBoYo2010, Yo2004} and references therein).

Despite its simplicity, determining whether the set $\A$ is dominated using Definition \ref{def: domination (norm)} is a difficult endeavour. To alleviate some of these difficulties, Avila, Bochi and Yoccoz \cite{AvBoYo2010} showed that domination is equivalent to the fact that $\A$ acts as a uniformly contracting iterated function system on a part of the projective real line $\R\P^1$ (see Theorem \ref{aby} for the precise statement). A higher-dimensional analogue of this condition was obtained by Bochi and Gourmelon \cite{BoGo2009}, and was further explored by Barnsley and Vince \cite{BaVi2012}. Recently, these ideas have been adapted to the setting of operators on Banach spaces by Blumenthal and Morris \cite{BlMo2019}, and Quas, Thieullen and Zarrabi \cite{QTZ2019}. 

Inspired by this geometric characterization, we aim at providing a necessary and a sufficient condition for domination that explicitly features the geometry of the action of a pair of matrices on $\R\P^1$.

First, we recall that a non-identity matrix $A\in\SL$ is called \emph{hyperbolic} if $\lvert \tr(A)\rvert >2$, \emph{parabolic} if $\lvert \tr(A)\rvert =2$, and \emph{elliptic} otherwise. A hyperbolic matrix $A$ has two eigenvalues $\lvert\lambda_u(A)\rvert>\lvert\lambda_s(A)\rvert$, with distinct corresponding eigenvectors $u(A),s(A)\in \R\P^1$. On the other hand, if $A$ is parabolic then it has a unique eigenvector with algebraic multiplicity two, in which case we write $u(A)=s(A)$.

Throughout the article, we identify $\R\P^1$ with the extended real line $\R\cup\{\infty\}$. With this identification in mind, for a pair of hyperbolic matrices $A,B\in\SL$ we define their \emph{cross-ratio} as
\[
C(A,B)\vcentcolon=[u(A),s(A),u(B),s(B)]=\frac{u(A)-u(B)}{u(A)-s(B)}\frac{s(A)-s(B)}{s(A)-u(B)},
\]
with the standard conventions about the point at infinity. If $A,B$ are hyperbolic matrices with pairwise distinct fixed points then $C(A,B)$ is well-defined (i.e. both denominators are non-zero) and does not equal zero or one. We also note that cross-ratios may appear with different orderings in the literature. 

The cross-ratio, which has been previously considered by Jacques and Short \cite{JS}, is a simple way to quantify the geometric configuration of the pair $A,B$. For example, if $C(A,B)$ is well-defined and $C(A,B)<1$, then $A$ and $B$ exhibit behaviour similar to positive matrices, in that they map a non-trivial open and connected subset of $\R\P^1$ compactly inside itself, meaning that $\{A,B\}$ is dominated due to the aforementioned result of Avila, Bochi and Yoccoz (see Section \ref{sect: cross} for more details).

Our first result provides an algorithm for constructing dominated sets, that depends on evaluating the cross-ratios of pairs.

\begin{theorem}\label{MAIN}
Let $u_1,u_2,\dots,u_N,s_1,s_2,\dots s_N\in\R\P^1$ be $2N$ pairwise distinct points in the projective real line. Suppose that $A_1,A_2,\dots A_N\in\SL$ are  hyperbolic matrices with $u(A_i)=u_i$ and $s(A_i)=s_i$, for $i=1,2,\dots,N$, and consider the constants 
\[
M_{i,j} = \max\left\{1+\left\lvert C(A_i,A_j)\right\rvert,1+\frac{1}{\left\lvert C(A_i,A_j)\right\rvert}\right\},
\]
for all pairs $i,j\in\{1,2,\dots,N\}$ with $i\neq j$. If 
\begin{equation}\label{eq: MAIN}
\tfrac{1}{2}\lvert \tr(A_k)\rvert > \max_{i\neq j}\{M_{i,j}\}\cdot \max_{i\neq j}\left\{\frac{\lvert C(A_i,A_j)\rvert +1}{\lvert C(A_i,A_j)-1\rvert}\right\}, \quad \text{for all}\ k=1,2,\dots,N,
\end{equation}
then the set $\{A_1,A_2,\dots, A_N\}$ is dominated.
\end{theorem}

Note that the constants $M_{i,j}$ and the constant on the right-hand side of \eqref{eq: MAIN}, are well-defined and depend only on the collection of pairwise distinct points $\{u_1,u_2,\dots,u_N,s_1,s_2,\dots s_N\}$. 

Observe that each hyperbolic matrix $A$ is determined by three independent parameters. Two of these are its eigenvectors, and the third is its \emph{contraction rate}; i.e. a measure of how much $A$ shifts the projective space towards its fixed point $u(A)$. This contraction rate is typically quantified by $\lvert \tr(A)\rvert$ (see also Section \ref{sect: hyperbolic geometry} for a more geometric description).

So, Theorem \ref{MAIN} tells us that if we wanted to construct a dominated set $\{A_1,A_2,\dots,A_N\}$, where each matrix $A_i$ has some prescribed eigenvectors $u_i,s_i$, then we only need increase the contraction rates of all $A_1,A_2,\dots,A_N$ enough to satisfy the open property \eqref{eq: MAIN}. In particular, the algebraic simplicity of the constants in \eqref{eq: MAIN} makes Theorem \ref{MAIN} an efficient tool for constructing explicit examples of dominated sets, which are quite rare in the literature. 

The proof of Theorem \ref{MAIN} also shows that when \eqref{eq: MAIN} is satisfied, the \emph{group} generated, under multiplication, by the matrices $\{A_1,A_2,\dots,A_N\}$ is a discrete, free and purely hyperbolic subgroup of $\SL$. Groups of this type are called \emph{(classical) Schottky groups}, and appear naturally in the study of Riemann surfaces and Teichm\"uller theory (see, for example, \cite{Ma1998,He2015} and references therein). Therefore, Theorem \ref{MAIN} also provides an efficient algorithm for producing Schottky groups with prescribed geometry, in the same spirit as the famous Poincar\'e's Theorem (see \cite[Section 9.8]{Be1995}).
 
The techniques involved in this article revolve around analysing the constraints on the traces of a dominated pair $\{A,B\}$, imposed by the geometric configuration of $A$ and $B$ (as quantified by $C(A,B)$).

In particular, we are led to to the following dichotomy, which shows that under a certain threshold the dynamical behaviour of $\{A,B\}$ is either very rigid, or quite chaotic.

\begin{theorem}\label{thm: anti classification}
Assume that $A,B\in\SL$ are hyperbolic matrices such that $C(A,B)$ is well-defined and satisfies $C(A,B)>1$. If
\[
\max\left\{\tfrac{1}{2}\lvert \mathrm{tr}(A)\rvert, \tfrac{1}{2}\lvert \mathrm{tr}(B)\rvert\right\}\leq \ \frac{5\ C(A,B)-1}{3\ C(A,B)+1},
\]
then the semigroup, under matrix multiplication, generated by $\{A,B\}$ is either a discrete subgroup of $\SL$, or dense in $\SL$.
\end{theorem}

Definition \ref{def: domination (norm)} immediately shows that if a finite set $\A$ is dominated, then the semigroup generated by $\A$ is a discrete subset of $\SL$ that contains only hyperbolic matrices. Therefore, a direct corollary of Theorem \ref{thm: anti classification} is the following necessary condition for domination.

\begin{corollary}\label{coro: jorg}
If the set $\{A,B\}\subset\SL$ is dominated, then $C(A,B)$ is well-defined, and either $C(A,B)<1$ or 
\[
\max\left\{\tfrac{1}{2}\lvert \mathrm{tr}(A)\rvert, \tfrac{1}{2}\lvert \mathrm{tr}(B)\rvert\right\}> \ \frac{5\ C(A,B)-1}{3\ C(A,B)+1}.
\]
\end{corollary}

Corollary \ref{coro: jorg} is in the same spirit as J\"orgensen's inequality \cite{Jo1976} from the theory of Kleinian groups, as it quantifies the the intuitive fact that if a pair is dominated then the contraction rates of the matrices cannot both be very small, unless the geometry of their action is somewhat trivial (i.e. $C(A,B)<1$). 

J\"orgensen's inequality has been extended to a large class of matrix sets, that contains all dominated sets by Jacques and Short \cite[Theorem 12.11]{JS}, and was later used in \cite{Ch2022} for the study of the boundary of domination. It is also worth mentioning that a condition fulfilling a similar purpose was obtained by Avila, Bochi and Yoccoz in \cite[Lemma 4.8]{AvBoYo2010}. The main advantage of Corollary \ref{coro: jorg} is its very explicit dependence on the geometric configuration of $\{A,B\}$, as is evident by the use of the cross-ratio.

Finally, our techniques allow us to obtain a necessary and sufficient condition for domination in the special case of a pair of matrices with equal traces.

\begin{theorem}\label{thm: equal trace}
Assume that $A,B\in\SL$ are hyperbolic matrices such that $C(A,B)$ is well-defined, $C(A,B)>1$ and $\lvert \tr(A) \rvert =\lvert \tr(B) \rvert=t$. Then $\{A,B\}$ is dominated if and only if ${\tfrac{1}{2}t>\sqrt{C(A,B)}}$.
\end{theorem}

In order to prove our results, we employ techniques from two-dimensional hyperbolic geometry in the complex plane. In particular, we view $\SL$-matrices as M\"obius transformations acting on the hyperbolic plane, and study how this action affects their domination properties. This perspective allows us to better control the geometric configuration of a finite set of matrices in order to obtain the explicit bounds present in our theorems.

The article is structured as follows: In the second section we introduce the concepts necessary for our work (M\"obius transformations and hyperbolic geometry) and thoroughly analyse cross-ratios. Theorem \ref{thm: anti classification} is proved in Section \ref{sect: necessary condition}. Section \ref{sect: constraints on pairs} is dedicated to establishing certain constraints on pairs of matrices that guarantee domination, and contains the proof of Theorem \ref{thm: equal trace}. Section \ref{sect: proof of main theorem} is dedicated to the proof of our example-constructing method, Theorem \ref{MAIN}, and its analogue for Schottky groups.

\section{Background material}

\subsection{Domination and the action on projective space}

As mentioned in the introduction, we identify the real projective line $\R\P^1$ with the extended real line $\overline{\R}\vcentcolon=\R\cup\{\infty\}$. We  will frequently refer to \emph{intervals} of the extended real line, which are the usual intervals of $\R$ along with sets of the form $[\infty,a)\cup(b,\infty]$, for some real numbers $a\leq b$, and usual modifications to include one of (or both) $a$ and $b$.

With this convention $\SL$-matrices act on $\R\P^1$ as M\"obius transformations with real entries. So, we have the following correspondence
\[
\PSL\ni \begin{pmatrix}
a & b \\ c & d
\end{pmatrix} \mapsto \frac{ax+b}{cx+d}=f(x).
\]
The use of the group $\PSL=\SL/\{\pm I\}$ instead of $\SL$ is justified by the fact that the matrices $A$ and $-A$ induce the same M\"obius transformation. So for the rest of this article we will consider finite sets of M\"obius transformations $\{f_1,f_2,\dots, f_n\}\subset \PSL$ instead of matrices. With this correspondence, the product of two matrices $A_1\cdot A_2$ corresponds to the composition $f_1\circ f_2$, where $f_1$ is the M\"obius transformation induced by $A_1$ and similarly for $f_2$. Also, we equip $\PSL$ with the quotient topology induced by the topology of $\SL$. 

Observe that if a set $\{A_1,A_2,\dots, A_N\}\subset\SL$ is dominated, then any set of the form $\{\pm A_1,\pm A_2,\dots, \pm A_N\}$ is also dominated. Therefore, Definition \ref{def: domination (norm)} carries over verbatim to finite subsets of $\PSL$. In $\PSL$, however, we also have a geometric characterization of domination due to Avila, Bochi and Yoccoz \cite[Theorem 2.2]{AvBoYo2010}, which is a key element of our analysis.

\begin{theorem}\label{aby}
A set $\{f_1,f_2,\dots,f_n\}\subset\PSL$ is dominated if and only if there exists a finite union $X\subsetneq\overline{\R}$ of open intervals, with disjoint closures, such that $\overline{f_i(X)}\subset X$, for all $i=1,2,\dots,n$.
\end{theorem}

\subsection{Hyperbolic geometry of the M\"obius group}\label{sect: hyperbolic geometry}

One of the main advantages of considering $\PSL$ as a group of M\"obius transformations is that its action on $\overline{\R}$ can be extended to the upper half-plane $\H=\{z\in\C\colon \mathrm{Im}\ z >0\}$ of the complex plane. If we now equip $\H$ with the hyperbolic distance $d_\H$, i.e. the distance induced by the Riemannian metric $\lvert dz \rvert/\Im{z}$, then $\PSL$ is exactly the group of conformal isometries of the complete metric space $(\H,d_\H)$. A thorough analysis of the hyperbolic geometry of $\H$, along with most of the material presented in this subsection, can be found in \cite[Chapter 7]{Be1995}.

One can compute the following closed formula for $d_\H$ (see, for example, \cite[Theorem 7.2.1]{Be1995}):

\begin{equation}\label{eq: hyperbolic metric in H}
d_\H(z_1,z_2)=\log\frac{\lvert z_1 - \overline{z_2}\rvert + \lvert z_1- z_2\rvert}{\lvert z_1 - \overline{z_2}\rvert - \lvert z_1- z_2\rvert},\quad \text{for all}\ z_1,z_2\in\H,
\end{equation}

or, equivalently

\begin{equation}\label{dist-formula}
\tanh\tfrac{1}{2}d_\H(z_1,z_2) = \frac{\lvert z_1-z_2\rvert}{\lvert z_1-\overline{z_2}\rvert}.
\end{equation}

The geodesics of the hyperbolic distance $d_\H$ are the vertical half-lines and semicircles orthogonal to $\R$. We say that two hyperbolic geodesics are \emph{disjoint} if they do not meet in $\H$ and have distinct endpoints i.e. their closures are disjoint subsets of $\H\cup\overline{\R}$. The distance between two geodesics $\gamma_1$ and $\gamma_2$ is defined as 
\[
d_\H(\gamma_1,\gamma_2)\vcentcolon=\inf\{d_\H(z_1,z_2)\colon z_1\in\gamma_1,\ z_2\in\gamma_2\}.
\]

Now suppose that $f(x)=(ax+b)/(cx+d)$ is the M\"obius transformation induced by a matrix $A\in\SL$. We define the trace of $f$ as $\lvert\tr(f) \rvert \vcentcolon=\lvert \tr(A)\rvert = \lvert a+ d\rvert$. The use of the absolute value here is important since it alleviates any ambiguity in the sign. With this definition the classification of matrices to hyperbolic, parabolic and elliptic we mentioned in the introduction extends naturally to M\"obius transformations. That is, $f$ is called \emph{hyperbolic} if the inducing matrix $A$ is hyperbolic, and similarly for \emph{parabolic} and \emph{elliptic} transformations. 

An equivalent classification can be obtained by considering the fixed points in $\H\cup\overline{\R}$ of the transformation. So, a non-identity M\"obius transformation is hyperbolic, parabolic or elliptic depending on whether it has two fixed points $\overline{\R}$, one fixed point in $\overline{\mathbb{R}}$, or one fixed point in $\H$, respectively.

Suppose that $f$ is a hyperbolic transformation in $\PSL$. The fixed points of $f$ will be denoted by $u(f),s(f)$ and are exactly the projections in $\overline{\R}$ of the eigenvectors of the matrix inducing $f$ (hence the similar notation). If we let $f^n\vcentcolon=f\circ f\circ \dots \circ f$ denote the $n$-th iterate of $f$, i.e. the composition of $f$ with itself $n$-times, then the fixed point $u(f)$ of $f$ has the property that the sequence $(f^n)$ converges to the constant $u(f)$ uniformly on compact subsets of $\H$. We thus call $u(f)$ \emph{the attracting fixed point} of $f$. The other fixed point is called the \emph{repelling fixed point}.

The unique hyperbolic geodesic joining the fixed points $u(f),s(f)$ of a hyperbolic transformation $f$ is called the \emph{axis} of $f$ and is denoted by $\mathrm{Ax}(f)$. Since $f$ is an isometry of $d_\H$, the curve $\mathrm{Ax}(f)$ is preserved set-wise by $f$.

All our results require a quantifiable way to measure the contraction rate of hyperbolic transformation. The prevalent method to do so is the \emph{translation length} $T_f$ of a hyperbolic transformation $f$, defined as follows
\[
T_f\vcentcolon=\inf\{d_\H(z,f(z))\colon z\in \H\}.
\]
This infimum is in fact a minimum that is attained for any point $z_0\in\mathrm{Ax}(f)$. The translation length is directly related to the trace of the transformation with the following formula, found in \cite[Section 7.34]{Be1995}.
\begin{equation}\label{eq: trace and translation length}
\cosh\tfrac{1}{2}T_f= \tfrac{1}{2}\lvert \tr(f)\rvert, 
\end{equation}
meaning that a simple way to quantify the contraction rate is by the trace, as we mentioned in the introduction. 

Furthermore, observe that since $f$ fixes $\mathrm{Ax}(f)$ set-wise, $f^n(z_0)\in\mathrm{Ax}(f)$ for any $z_0\in\mathrm{Ax}(f)$ and all $n\in\N$. Using the facts that $\mathrm{Ax}(f)$ is a hyperbolic geodesic and $f$ is an isometry we get that
\begin{align}
T_{f^n}&=d_\H(f^n(z_0),z_0)=d_\H(f^n(z_0),f^{n-1}(z_0))+d_\H(f^{n-1}(z_0),f^{n-2}(z_0))+\cdots+d_\H(f(z_0),z_0)\nonumber\\
	&=d_\H(f(z_0),z_0)+d_\H(f(z_0),z_0)+\cdots+d_\H(f(z_0),z_0)=n\ T_f. \label{eq: translation length of iterates}
\end{align}

We will often find it convenient to switch to the disc model $\mathbb{D}=\{z\in\C\colon \lvert z \rvert<1\}$ of hyperbolic space. In $\D$, the hyperbolic distance is induced by the metric $\frac{2\lvert dz \rvert}{1-\lvert z\rvert^2}$. Since the half-plane and the disc models are conformally equivalent, we can move freely between them and will do so without explicit mention. In particular, we use the disc model for most of our figures where a hyperbolic transformation will be portrayed by drawing its axis as a directed hyperbolic geodesic pointing towards its attracting fixed point as in Figure \ref{fig: hyperbolic transformation}.

\begin{figure}[ht]

\begin{tikzpicture}[scale=1.8, , decoration={
    markings,
    mark=at position 0.5 with {\arrow{To}}}]
\draw  (0.,0.) circle (1.cm);
\draw [postaction={decorate}, shift={(0.9653642003099734,-1.148898844767251)}]  plot[domain=1.5402527146681928:2.9989546899046022,variable=\t]({1.*1.118881850218172*cos(\t r)+0.*1.118881850218172*sin(\t r)},{0.*1.118881850218172*cos(\t r)+1.*1.118881850218172*sin(\t r)});

\node[above left] at (0.27,-0.27) {$f$};

\filldraw[black] (-0.1421547799234557,-0.9898444415891388) circle (.5pt) node[below] {$u(f)$};
\filldraw[black] (0.9995335801415512,-0.030538863263278673) circle (.5pt) node[right] {$s(f)$};
\end{tikzpicture}
\caption{A hyperbolic transformation depicted by its axis directed towards the attracting fixed point.}
\label{fig: hyperbolic transformation}
\end{figure}

\subsection{Cross-ratios}\label{sect: cross}

Here we explore how the cross-ratio of two hyperbolic transformations can be used to determine their geometric configuration. First, recall that the cross-ratio of two hyperbolic transformations $f,g\in\PSL$ is given by
\begin{equation}\label{eq: cross-ratio def}
C(f,g)\vcentcolon=[u(f),s(f),u(g),s(g)]=\frac{u(f)-u(g)}{u(f)-s(g)}\frac{s(f)-s(g)}{s(f)-u(g)}.
\end{equation}
If any of the fixed points of $f$ and $g$ is the point at infinity, then we simply omit any of the differences in the above fractions that contain it. For example, if $u(g)=\infty$, then 
\[
C(f,g)=\frac{s(f)-s(g)}{u(f)-s(g)}.
\]
Note that the cross-ratio is well-defined only when $u(f)\neq s(g)$ and $u(g)\neq s(f)$. If, however, any of these two conditions does not hold, then it is easy to show that the pair $\{f,g\}$ is not dominated (see \cite[Lemma 10.4]{JS}). So, since we are only interested in the domination properties of sets of transformations, we always assume that the cross-ratio of a pair is well-defined.

Simple calculations show that $C(f,g)=1$ is equivalent to $(u(f)-s(f))(u(g)-s(g))=0$. Therefore, for any two hyperbolic transformations we have that $C(f,g)\neq 1$. Furthermore, if $f$ and $g$ do not share any fixed points we have the following standard properties for the cross-ratio: $C(f,g)=C(g,f)$ and $C(f^{-1},g)\cdot C(f,g)=1$.

The value of the cross-ratio leads to a classification of the geometric configuration of the axes of two hyperbolic transformations, as shown in Figure \ref{cross} (this classification is justified by Lemma \ref{crl}, to follow). This figure, along with the geometric characterization of domination in Theorem \ref{aby}, immediately show that if $C(f,g)<1$, then $\{f,g\}$ is dominated, as we mentioned in the introduction.

\begin{figure}[ht]
\centering
\begin{tikzpicture}

\begin{scope}[scale=1.5, decoration={
    markings,
    mark=at position 0.8 with {\arrow{To}}}]
    
\draw [line width=0.4pt] (0.,0.) circle (1.cm);
\draw [line width=0.4pt, postaction= {decorate}] (0.7071067811865475,0.7071067811865475) -- (-0.7071067811865476,-0.7071067811865476);
\draw [line width=0.4pt, postaction= {decorate}] (-0.7071067811865477,0.7071067811865475) --  (0.7071067811865476,-0.7071067811865474);
\node[below] at (0,-1.1) {$C(f,g)\leq 0$};
\node [left] at (-0.40,-0.37) {$f$};
\node [right] at (0.41,-0.38) {$g$};
\end{scope}

\begin{scope}[xshift =4.5cm, scale=1.5, decoration={
    markings,
    mark=at position 0.53 with {\arrow{To}}}]
    
\draw [line width=0.4pt] (0.,0.) circle (1.cm);
\draw [shift={(-1.8312065528259258,0.)},line width=0.4pt, postaction={decorate}]  plot[domain=0.5776873665141405:-0.5776873665141409,variable=\t]({1.*1.5340526194080204*cos(\t r)+0.*1.5340526194080204*sin(\t r)},{0.*1.5340526194080204*cos(\t r)+1.*1.5340526194080204*sin(\t r)});
\draw [shift={(1.8312065528259265,0.)},line width=0.4pt, postaction={decorate}]  plot[domain=2.563905287075653:3.719280020103933,variable=\t]({1.*1.5340526194080215*cos(\t r)+0.*1.5340526194080215*sin(\t r)},{0.*1.5340526194080215*cos(\t r)+1.*1.5340526194080215*sin(\t r)});

\node[below] at (0,-1.1) {$0<C(f,g)<1$};
\node [left] at (-.4,0) {$f$};
\node [right] at (.4,0) {$g$};
\end{scope}

\begin{scope}[xshift =9cm, scale=1.5, decoration={
    markings,
    mark=at position 0.523 with {\arrow{To}}}]
    
\draw [line width=0.4pt] (0.,0.) circle (1.cm);
\draw [shift={(-1.8312065528259258,0.)},line width=0.4pt, postaction={decorate}]  plot[domain=-0.5776873665141409:0.5776873665141405,variable=\t]({1.*1.5340526194080204*cos(\t r)+0.*1.5340526194080204*sin(\t r)},{0.*1.5340526194080204*cos(\t r)+1.*1.5340526194080204*sin(\t r)});
\draw [shift={(1.8312065528259265,0.)},line width=0.4pt, postaction={decorate}]  plot[domain=2.563905287075653:3.719280020103933,variable=\t]({1.*1.5340526194080215*cos(\t r)+0.*1.5340526194080215*sin(\t r)},{0.*1.5340526194080215*cos(\t r)+1.*1.5340526194080215*sin(\t r)});
\node[below] at (0,-1.1) {$C(f,g)>1$};
\node [left] at (-.4,0) {$f$};
\node [right] at (.4,0) {$g$};
\end{scope}

\end{tikzpicture}
\caption{Cross-ratio and geometric configuration}\label{cross}
\end{figure}

Now, let us establish formulas that relate the cross-ratio $C(f,g)$ of two hyperbolic transformations $f$ and $g$ with the geometric configuration of their axes. Equivalent versions of the formulae in Lemma \ref{crl}, to follow, appear in \cite[p. 166]{Be1995}, but we include a proof for the sake of completeness.

We first need to mention that if the axes of two hyperbolic transformations $f$ and $g$ cross at a point $p\in\mathbb{H}$, we define the angle $\theta\in[0,\pi]$ between $\mathrm{Ax}(f)$ and $\mathrm{Ax}(g)$ to be the angle at $p$ of the hyperbolic triangle defined by $\alpha(f),\alpha(g)$ and $p$. In this case, we say that the axes of $f$ and $g$ \emph{cross at an angle $\theta$}.

\begin{lemma}\label{crl}
Suppose that $f$ and $g$ are hyperbolic transformations.
\begin{enumerate}
\item If the axes of $f$ and $g$ cross at an angle $\theta$, then $C(f,g)=-\tan^2\tfrac{1}{2}\theta$.

\item If the axes of $f$ and $g$ are disjoint and a hyperbolic distance $d$ apart, then
\[
C(f,g)= 
\begin{cases}
\tanh^2\tfrac{1}{2}d, &\text{if}\quad 0<C(f,g)<1,\\[5pt]
\coth^2\tfrac{1}{2}d, &\text{if}\quad C(f,g)>1.
\end{cases}
\]
\end{enumerate}
\end{lemma}

\begin{proof}
For part \textit{(i)}, suppose that the axes of $f$ and $g$ cross at an angle $\theta$. Conjugate $f$ and $g$ by a transformation $\phi\in\mathrm{PSL}(2,\mathbb{C})$ so that they act on the unit disc. If we denote $F=\phi\circ f \circ \phi^{-1}, G=\phi\circ g \circ \phi^{-1}$, then $C(F,G)=C(f,g)$ by the invariance of the cross-ratio, and their axes still cross at an angle $\theta$ by the conformality of M\"obius transformations. Moreover, $\phi$ can be chosen so that the axes of $F$ and $G$ meet at the origin, and the diameter through $i$ and $-i$ bisects the angle $\theta$ (see Figure~\ref{cross-for} on the left). Due to the symmetry of this configuration we have that $u(F)=\overline{s(G)},\ u(F)=-s(F)$ and $u(G)=-s(G)$. Hence,
\[
C(F,G)=\frac{u(F)-u(G)}{u(F)-s(G)}\frac{s(F)-s(G)}{s(F)-u(G)}= \frac{-2\Re s(F)}{-2i\Im s(F)} \frac{2\Re s(F)}{2i\Im s(F)}= -\left(\frac{\Re s(F)}{\Im s(F)}\right)^2.
\]
The law of cosines applied to the Euclidean triangles with vertices $0, u(F),s(G)$ and $0,u(F),u(G)$, yield $\left(2\Re u(F)\right)^2 = 2(1-\cos\theta )$ and  $\left(2\Im u(F)\right)^2=2(1-\cos(\pi-\theta))$, respectively. Therefore,
\[
C(F,G) = -\frac{1-\cos\theta}{1-\cos(\pi-\theta)}=-\frac{1-\cos\theta}{1+\cos\theta}=-\tan^2\tfrac{1}{2}\theta.
\]
Moving on to part (ii), assume that the axes of $f$ and $g$ are disjoint and a hyperbolic distance $d$ apart. Since $C(f^{-1},g)=1/C(f,g)$ and $C(f,g)\neq 1$, it suffices to show that $C(f,g)=\tanh^2\tfrac{1}{2}d$ whenever $0<C(f,g)<1$. Conjugate $f$ and $g$ in $\PSL$ so that $f$ fixes $-\sigma, \sigma$ and $g$ fixes $-\lambda, \lambda$, for some $0<\sigma<\lambda$, as shown in Figure~\ref{cross-for} on the right. Then,
\[
C(f,g)=\frac{(\lambda-\sigma)^2}{(\lambda+\sigma)^2}.
\]
Since the imaginary half-axis is perpendicular to both $\textrm{Ax}(f)$ and $\textrm{Ax}(g)$, we have that $d=\rho_{\mathbb{H}}(i\sigma,i\lambda)$. Using formula \eqref{dist-formula} yields
\[
\tanh^2\tfrac{1}{2}d=\tanh^2\tfrac{1}{2}\rho_{\mathbb{H}}(i\sigma,i\lambda)=\frac{\lvert i\sigma - i \lambda \rvert^2}{\lvert i\sigma -\overline{i\lambda} \rvert^2} = \frac{(\lambda-\sigma)^2}{(\lambda+\sigma)^2}=C(f,g),
\]
completing our proof.
\end{proof}

\begin{figure}[ht]
\centering
\begin{tikzpicture}

\begin{scope}[scale =2.5 , xshift= 3cm, decoration={
    markings,
    mark=at position 0.635 with {\arrow{To}}}
    ] 
\draw[->] (-1.1,0) -- (1.1,0);
\draw (0,0) -- (0,1.2);
\draw[postaction={decorate}] (5/8,0) arc (0:180:5/8);
\draw[postaction={decorate}] (5.5/6,0) arc (0:180:5.5/6);
\draw (-0,5.8/6) -| (-.05, 5.5/6);
\draw (-0,5.4/8) -| (-.05, 5/8);

\node[below] at (-0.2,5/9) {$f$};
\node[above] at (-0.35,5.2/6) {$g$};
\node[below] at (5.5/6,0) {$\lambda$};
\node[below] at (-5.5/6,0) {$-\lambda$};
\node[below] at (5/8,0) {$\sigma$};
\node[below] at (-5/8,0) {$-\sigma$};
\node[right] at (0.1, 0.7708) {$d$};
\draw[decorate,decoration={brace,raise=4pt}] (0,5.45/6) --  (0,5.05/8);
\end{scope}

\begin{scope}[scale=1.9, yshift=.8cm, decoration={
    markings,
    mark=at position 0.8 with {\arrow{To}}}
    ] 
\coordinate (o) at (0,0);
\coordinate (r) at (0.71,-0.71);
\coordinate (l) at (-0.71,-0.71);

\draw [line width=0.4pt] (0.,0.) circle (1.cm);
\draw [line width=0.4pt, postaction= {decorate}] (0.7071067811865475,0.7071067811865475) -- (-0.7071067811865476,-0.7071067811865476) ;
\draw [line width=0.4pt, postaction= {decorate}] (-0.7071067811865477,0.7071067811865475) --  (0.7071067811865476,-0.7071067811865474);
\pic [draw, "$\theta$", angle eccentricity=1.5, line width=0.4pt] {angle = l--o--r};
\node[left] at (-0.40,-0.37) {$F$};
\node[right] at (0.41,-0.38) {$G$};
\node[above] at (0,0.05) {$0$};
\node [below left] at (-0.71,-0.71) {$u(F)$};
\node [below right] at (0.71,-0.71) {$u(G)$};
\node [above right] at (0.71,0.71) {$s(F)$};
\node [above left] at (-0.71,0.71) {$s(G)$};

\end{scope}
\end{tikzpicture}
\caption{Formulas for the cross-ratio in Lemma \ref{crl}.}\label{cross-for}
\end{figure}

Using the standard formulas for the trigonometric and hyperbolic trigonometric functions we can easily obtain the next formulas that will prove useful to our calculations.

\begin{corollary}\label{cross-coro}
Suppose that $f$ and $g$ are hyperbolic transformations.
\begin{enumerate}
\item If the axes of $f$ and $g$ cross at an angle $\theta$, then $\cos^2\tfrac{1}{2}\theta=\frac{1}{1-C(f,g)}$.

\item If the axes of $f$ and $g$ are disjoint and a hyperbolic distance $d$ apart, then $\cosh d = \frac{C(f,g)+1}{\lvert C(f,g)-1\rvert}$.

\end{enumerate}
\end{corollary}

\section{A necessary condition for domination} \label{sect: necessary condition}

In this section we prove Theorem \ref{thm: anti classification} which immediately yields the necessary condition for matrix domination given in Corollary \ref{coro: jorg}.

The main ingredient for the proof is the following proposition, which shows that if the traces of two transformations are small, relative to their geometric configuration, then the semigroup they generate contains elliptic transformations.

\begin{proposition}\label{anti}
Let $f,g$ be two hyperbolic transformations with $C(f,g)>1$. If 
\[
\max\left\{\tfrac{1}{2}\lvert \mathrm{tr}(f)\rvert, \tfrac{1}{2}\lvert \mathrm{tr}(g)\rvert\right\}\leq \ \frac{5C(f,g)-1}{3C(f,g)+1},
\]
then there exist $n,m\in\N$ so that $f^n\circ g^m$ is elliptic.
\end{proposition}

\begin{proof}
Since $C(f,g)>1$, the axes $\mathrm{Ax}(f)$ and $\mathrm{Ax}(g)$, of $f$ and $g$ respectively, are disjoint. Let $d>0$ be the hyperbolic distance between $\mathrm{Ax}(f)$ and $\mathrm{Ax}(g)$. The proof revolves around evaluating the trace of the composition of $f$ and $g$, which is given by the following equation found in \cite[Theorem 7.38.3]{Be1995}:
\begin{equation}\label{eq: antiparallel trace eq}
\frac{1}{2}\lvert\text{tr}(f\circ g)\rvert=\lvert\cosh d\sinh\tfrac{1}{2}T_f\sinh\tfrac{1}{2}T_g-\cosh\tfrac{1}{2}T_f\cosh\tfrac{1}{2}T_g\rvert,
\end{equation}
where $T_f,T_g$ are the translation lengths of $f$ and $g$, respectively. Let $\mathbb{R}_+^2$ be the first quadrant of $\mathbb{R}^2$, endowed with the Euclidean topology, and consider the function $h\colon \R_+^2\to\R$, with 
\begin{equation}\label{eq: antiparallel def of h}
h(x,y)=\cosh d\sinh x\sinh y-\cosh x\cosh y.
\end{equation}
So, equation \eqref{eq: antiparallel trace eq} can be rewritten as $\frac{1}{2}\lvert\text{tr}(f\circ g)\rvert = \rvert h\left(\tfrac{1}{2}T_f,\tfrac{1}{2}T_g\right)\rvert$. Also, from \eqref{eq: translation length of iterates} we have that $T_{f^n}=n\ T_f$, meaning that
\begin{equation}\label{n,m-traces}
\frac{1}{2}\lvert\text{tr}(f^n\circ g^m)\rvert=\left\lvert h\left(\tfrac{n}{2}T_f,\tfrac{m}{2}T_g\right) \right\rvert, \quad \text{for all}\ n,m\in\N.
\end{equation}
Define the open set $D=\left\{(x,y)\in\mathbb{R}_+^2\colon -1<h(x,y)<1\right\}$. Since a transformation $h$ is elliptic if and only if $\lvert \mathrm{tr}(h)\rvert <2$, we have that the composition $f^n\circ g^m$ is elliptic, for some $n,m\in\N$, if and only if $\left\lvert h\left(\tfrac{n}{2}T_f,\tfrac{m}{2}T_g\right) \right\rvert<1$, or equivalently $\left(\tfrac{n}{2}T_f,\tfrac{m}{2}T_g\right)\in D$.\\
So, our goal is to show that if 
\[
\max\left\{\tfrac{1}{2}\lvert \mathrm{tr}(f)\rvert, \tfrac{1}{2}\lvert \mathrm{tr}(g)\rvert\right\}\leq \ \frac{5C(f,g)-1}{3C(f,g)+1},
\]
then there exist integers $n,m\geq1$, so that $\left(\tfrac{n}{2}T_f,\tfrac{m}{2}T_g\right)\in D$.\\
Let $a,b>0$ be such that
\begin{align}\label{sinh}
&\sinh b =\frac{\sqrt{2}}{\sqrt{\cosh d-1}} &\text{and} & &\sinh a= \frac{\sqrt{2}}{(2\cosh d+1)\sqrt{\cosh d-1}}.
\end{align}
These are well-defined since $d>0$. Also, \eqref{sinh} are equivalent to
\begin{align}\label{cosh}
&\cosh b =\sqrt{\frac{\cosh d+1}{\cosh d-1}} &\text{and} & &\cosh a= \frac{(2\cosh d-1)}{(2\cosh d+1)}\sqrt{\frac{\cosh d+1}{\cosh d-1}}.
\end{align}
Observe that $b>a$, since $\sinh b > \sinh a$. Let $S$ be the square with vertices $(a,b), (b,b), (b,a), (a,a)$ and $C$ the open region in $\mathbb{R}^2_+$ bounded by $S$. That is, $\overline{C}=S\cup C$. Our main task is to show that $\overline{C}\setminus\{(a,b), (b,b), (b,a)\}\subset D$, as shown in Figure~\ref{square}.\\
\begin{figure}[h]
\centering
\begin{tikzpicture}
    \node[anchor=south west,inner sep=0] at (0,0) {\includegraphics[scale=3.2]{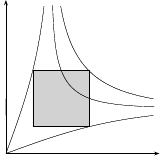}};
    \node[above right] at (4.825,4.825) {$(b,b)$};
    \node[above left, inner sep=1pt] at (1.80061319411490,4.82185316399136) {$(a,b)$};
	\node[below right, inner sep=1pt] at (4.82185316399136, 1.80061319411490) {$(b,a)$};
    \node[below left] at (1.80061319411490,1.80061319411490) {$(a,a)$};
    
	\node[below right] at (2.25,5.5) {$D$};    
	\node at (3.31123317905313,3.31123317905313) {$C$};
	
	\node[above, rotate=-17.2] at (7,3.6) {$h(x,y)=1$};
	\node[above, rotate=-7] at (6.5,2.93) {$h(x,y)=0$};
	\node[below, rotate=8] at (7.3,2.3) {$h(x,y)=-1$};
	\node[above, rotate=82] at (2.3,7.3) {$h(x,y)=-1$};
	
	\node at (4.82185316399136, 1.80061319411490)[circle,fill,inner sep=1pt]{};    
    \node at (1.80061319411490,4.82185316399136)[circle,fill,inner sep=1pt]{};
    \node at (1.80061319411490,1.80061319411490)[circle,fill,inner sep=1pt]{};
    \node at (4.82185316399136,4.82185316399136)[circle,fill,inner sep=1pt]{};
\end{tikzpicture}
\caption{The regions $C$ and $D$}\label{square}
\end{figure}
We start by showing that the edges of $S$, without the vertices $(a,b), (b,b), (b,a)$, are contained in $D$. We shall make heavy use of the fact that the line $y=x$ is an axis of symmetry for $h$; i.e. $h(x,y)=h(y,x)$, for all $x,y>0$.\\
Using the formula for $h$ in \eqref{eq: antiparallel def of h} and equations \eqref{sinh}, \eqref{cosh} we can evaluate $h$ on the vertices of $S$ as follows:
\begin{align}
h(a,a)&=\cosh d \sinh^2a-\cosh^2a=(\cosh d -1)\sinh^2a-1=\frac{2}{(2\cosh d+1)^2}-1, \label{eq: antiparallel haa}\\
h(b,b)&= \cosh d \sinh^2b - \cosh^2b= \cosh d \frac{2}{\cosh d-1} - \frac{\cosh d +1}{\cosh d-1}=1,\label{eq: antiparallel hbb}
\end{align}
and
\begin{align}\label{hab}
h(b,a)=h(a,b) &= \cosh d \sinh a \sinh b - \cosh a \cosh b=\frac{2\cosh d - (2\cosh d-1)(\cosh d+1)}{(2\cosh d+1)(\cosh d-1)}=\nonumber \\
&=\frac{-2\cosh^2 d+\cosh d+1}{(2\cosh d+1)(\cosh d-1)}=-1.
\end{align}
Observe that \eqref{eq: antiparallel haa} implies that 
\begin{equation}\label{eq: antiparallel haa 2}
-1<h(a,a)<-\frac{7}{9}.
\end{equation}
Now consider the partial derivative
\begin{equation}\label{increasing}
\frac{\partial h}{\partial x}(x,b) = \cosh d \cosh x \sinh b - \sinh x \cosh b, \quad \text{for} \quad x>0.
\end{equation}
Equations \eqref{sinh} and \eqref{cosh} yield
\[
\frac{1}{\cosh d}\ \frac{\cosh b}{\sinh b}=\frac{\sqrt{\cosh d +1}}{\sqrt{2}\ \cosh d}<1<\frac{\cosh x}{\sinh x},\quad \text{for all}\: x>0.
\]
Applying this inequality to \eqref{increasing}, we obtain that $h$ is strictly increasing on the horizontal line $y=b$. This implies that 
\[
-1=h(a,b)<h(x,b)<h(b,b)=1,\quad \text{for any}\: x\in(a,b),
\]
meaning that $h(x,b)\in D$, for all $x\in(a,b)$. This proves that the horizontal edge of $S$ joining $(a,b)$ and $(b,b)$ (without these vertices) is contained in $D$. By the symmetry of $h$ we also obtain that the vertical edge of $S$ joining $(b,a)$ and $(b,b)$ (excluding the vertices again) is contained in $D$.\\
We now turn to the vertical edge of $S$ joining $(a,a)$ and $(a,b)$. Consider the partial derivative
\[
\frac{\partial h}{\partial y}(a,y) = \cosh d \sinh a \cosh y - \cosh a \sinh y, \quad \text{for}\quad y>0.
\]
Since $\frac{\partial h}{\partial y}(a,a)= (\cosh d -1) \sinh a \cosh a>0$, by continuity there exists $y_+\in(a,b)$, close enough to $a$, so that $\frac{\partial h}{\partial y}(a,y_+)>0$.\\
Suppose that $\frac{\partial h}{\partial y}(a,y)\neq0$ for all $y\in(a,b)$. Then, again by continuity and the fact that $\frac{\partial h}{\partial y}(a,y_+)>0$, we have that $\frac{\partial h}{\partial y}(a,y)>0$,  for all $y\in(a,b)$, which implies that $h(a,y)$ is increasing in $(a,b)$. But $h(a,a)>-1=h(a,b)$, from \eqref{hab} and \eqref{eq: antiparallel haa 2}, which is a contradiction.\\
We conclude that $\frac{\partial h}{\partial y}(a,y_0)=0$, for some $y_0\in(a,b)$. That is
\begin{align*}
&\frac{\partial h}{\partial y}(a,y_0)=\cosh d \sinh a \sinh y_0 - \cosh a \sinh y_0=0 & \iff& &\coth y_0 = \frac{\coth a}{\cosh d }.
\end{align*}
If $0<y<y_0$, then
\begin{align*}
&\coth y >\coth y_0 = \frac{\coth a}{\cosh d } & &\iff& &\cosh d \cosh y \sinh a > \cosh a \sinh y & &\iff& &\frac{\partial h}{\partial y}(a,y)>0.
\end{align*}
Similarly we can show that if $y>y_0$, then $\frac{\partial h}{\partial y}(a,y)<0$. Therefore the maximum of $h(a,y)$, for all $y>0$, is attained only at $y_0$. This means that 
\[
\min\{h(a,a), h(a,b)\} < h(a,y) \leq h(a,y_0),\quad \text{for all}\ y\in(a,b).
\]
But, $h(a,a)>-1=h(a,b)$, again by \eqref{hab} and \eqref{eq: antiparallel haa 2}, and so $h(a,y)>-1$, for all $(a,b)$. In order to conclude that $h(a,y)\in D$, for all $y\in(a,b)$, it suffices to show that $h(a,y_0)<0$, which we now prove.
\begin{align*}
h(a,y_0)<0 &\iff \cosh d \sinh a \sinh y_0 <\cosh a \cosh y_0 \\
&\iff \cosh d < \coth a \coth y_0= \coth a\ \frac{\coth a}{\cosh d }\\
&\iff \cosh d <\coth a.
\end{align*}
But, equations \eqref{sinh} and \eqref{cosh} yield that
\[
\coth a  = \frac{(2\cosh d-1)\sqrt{\cosh d+1}}{\sqrt{2}}=\frac{(\cosh d + \cosh d -1)\sqrt{\cosh d+1}}{\sqrt{2}} >\frac{\cosh d\sqrt{2}}{\sqrt{2}}=\cosh d,
\]
since $d>0$. This concludes the proof of the fact that $h(a,y)\in D$, for all $y\in(a,b)$, which implies that the vertical edge of $S$ joining $(a,a)$ and $(a,b)$ (apart from the vertex $(a,b)$) lies in $D$. Using the symmetry of $h$, again, we obtain that the horizontal edge of $S$ joining $(a,a)$ and $(b,a)$ (apart from the vertex $(b,a)$) lies in $D$.\\
So, as we claimed, all edges of $S$ lie in $D$, apart from the vertices $(a,b),(b,b),(b,a)$, which lie on the boundary of $D$, since $h(a,b)=h(b,a)=-1$ and $h(b,b)=1$. With these facts in mind, we move on to the proof of the inclusion $C\subset D$.\\
First, a simple computation shows that the function $h(x,x)$ is strictly increasing for all $x>0$. Therefore, recalling \eqref{eq: antiparallel def of h} and \eqref{eq: antiparallel hbb}, we get that 
\begin{equation}\label{eq: anitparallel y=x}
-1=\lim_{t\to0^+}h(t,t)<h(x,x)<h(b,b)=1,\quad \text{for all}\quad 0<x<b.
\end{equation}
This means that $C\cap\{(x,x)\colon x>0\}\subset D$. Now, fix $x_0>0$ and consider the derivative
\begin{equation}\label{eq: antiparallel final der}
\frac{\partial h}{\partial y}(x_0,y)=\cosh d \sinh x_0 \cosh y- \cosh x_0 \sinh y.
\end{equation}
Since the hyperbolic tangent is an increasing function, we obtain that for $y<x_0$
\begin{equation}\label{eq: antiparallel final der 2}
\tanh y < \tanh x_0 < \cosh d \tanh x_0.
\end{equation}
Combining \eqref{eq: antiparallel final der} and \eqref{eq: antiparallel final der 2} yields that $h(x_0,y)$ is strictly increasing for all $y\in(0,x_0)$, and any $x_0>0$. So, if $x_0\in(a,b)$ then
\[
-1<h(x_0,a)<h(x_0,y)<h(x_0,x_0)<1,
\]
where the inequality $-1<h(x_0,a)$ follows from the fact that $(x_0,a)$ lies on an edge of $S$ and inequality $h(x_0,x_0)<1$ is \eqref{eq: anitparallel y=x} applied to $x_0$. We conclude that the vertical line segment joining $(x_0,a)$ and $(x_0,x_0)$, which is contained in $C$, lies in $D$ for all $x_0\in(a,b)$. Using the symmetry of $h$ one last time, we obtain that $C\subset D$, as claimed.\\
To finish the proof of the proposition, note that since the edges of $S$ have length $b-a$, if a point $(x,y)\in\R_+^2$ is such that $\max\{x,y\}<b-a$, then there exist integers $n,m\geq1$ such that $(nx,my)\in C\subset D$. Therefore, recalling equation \eqref{n,m-traces}, we have that if
\[
\max\left\{\tfrac{1}{2}T_f, \tfrac{1}{2}T_g \right\}< b-a,
\]
then there exist integers $n,m\geq1$ such that $f^n\circ g^m$ is elliptic. Next, assume that $(x,y)\in\R_+^2$ is such that $\max\{x,y\}=b-a$. If $x<y$ (or, similarly, $y<x$), then there exist $n,m\geq1$ so that $(nx,my)\in \overline{C}\setminus\{(a,b),(b,b),(b,a)\}\subset D$ and we are done once again. On the other hand, if $x=y=b-a$, then since $b-a<b$ we can use \eqref{eq: anitparallel y=x} to deduce that $-1<h(b-a,b-a)<1$, meaning that $(x,y)\in D$.\\
We have thus shown that if 
\[
\max\left\{\tfrac{1}{2}T_f, \tfrac{1}{2}T_g \right\}\leq b-a,
\]
then there exist positive integers $n,m\geq1$ so that $f^n\circ g^m$ is elliptic. As such, the proof of this proposition will now be complete upon showing that
\[
\max\left\{\tfrac{1}{2}T_f, \tfrac{1}{2}T_g \right\}\leq b-a \quad \iff \quad \max\left\{ \tfrac{1}{2}\lvert\mathrm{tr}(f)\rvert, \tfrac{1}{2}\lvert \mathrm{tr}(g)\rvert\right\}\leq \frac{5\ C(f,g)-1}{3\ C(f,g)+1}.
\]
Recall that from \eqref{eq: trace and translation length} we have that $\cosh \left(\tfrac{1}{2}T_f\right) = \tfrac{1}{2} \lvert \mathrm{tr}(f)\rvert$ and so we only need to show that
\begin{equation}\label{eq: antiparallel end goal}
\cosh (b-a) = \frac{5C(f,g)-1}{3C(f,g)+1}.
\end{equation}
But, $\cosh(b-a) = \cosh b\cosh a - \sinh b\sinh a$, and all quantities on the right-hand side are given by equations \eqref{sinh} and \eqref{cosh}. So,
\begin{align}
\cosh(b-a) &= \frac{(2\cosh d -1)}{(2\cosh d+1)}\frac{(\cosh d+1)}{(\cosh d-1)}  - \frac{2}{(2\cosh d+1)(\cosh d-1)}\nonumber \\
	&=\frac{2\cosh^2 d +\cosh d-3}{(2\cosh d+1)(\cosh d-1)}\nonumber \\
	&=\frac{2\cosh d+3}{2\cosh d+1}. \label{coshba}
\end{align}
Now, by Corollary~\ref{cross-coro} (ii) and the fact that $C(f,g)>1$, we have that $\cosh d = \frac{C(f,g)+1}{C(f,g)-1}$ which when substituted to \eqref{coshba} yields \eqref{eq: antiparallel end goal}.
\end{proof}

Some of the arguments in the proof of Proposition \ref{anti} can be used to prove the following result, which will be used in the proof of Theorem \ref{thm: equal trace} in the next section. 

\begin{lemma}\label{lem: equal trace lemma}
Assume that $f,g\in\PSL$ are hyperbolic transformations with $C(f,g)>1$, and such that $\lvert \tr(f) \rvert =\lvert \tr(g) \rvert=t$. If $\tfrac{1}{2}t\leq \sqrt{C(f,g)}$, then the transformation $f\circ g$ is either elliptic or parabolic.
\end{lemma}

\begin{proof}
Since $C(f,g)>1$ the axes of $f,g$ are disjoint and a hyperbolic distance $d>0$ apart. Also, from \eqref{eq: trace and translation length} and our assumption, we have that 
\[
\tfrac{1}{2}t=\tfrac{1}{2}\lvert \tr(f)\rvert =\cosh \tfrac{1}{2} T_f = \tfrac{1}{2}\lvert \tr(g) \rvert =\cosh \tfrac{1}{2} T_g.
\]
This means that $T_f=T_g$, so let $T$ be this common value.\\
As in the proof of Proposition \ref{anti}, we have that 
\begin{equation}\label{eq: equal trace eq1}
\frac{1}{2}\lvert \tr(f\circ g)\rvert= \lvert h\left(\tfrac{1}{2}T_f,\tfrac{1}{2}T_g\right)\rvert= \lvert h\left(\tfrac{1}{2}T,\tfrac{1}{2}T\right)\rvert,
\end{equation}
where $h(x,y)=\cosh d\sinh x\sinh y-\cosh x\cosh y$. In \eqref{eq: anitparallel y=x} we showed that $-1<h(x,x)<h(b,b)=1$, for all $x\in(0,b)$, where $b>0$ satisfies 
\begin{equation}\label{eq: equal trace eq2}
\cosh b =\sqrt{\frac{\cosh d+1}{\cosh d-1}}.
\end{equation}
So, \eqref{eq: equal trace eq1} tells us that if $\tfrac{1}{2}T\leq b$ we have that $f\circ g$ is either elliptic (when $T<b$) or parabolic (when $T=b$). Equivalently, by \eqref{eq: equal trace eq2}, the inequality $\tfrac{1}{2}T\leq b$ can be written as
\begin{equation}\label{eq: equal trace eq3}
\tfrac{1}{2}t=\cosh\tfrac{1}{2}T\leq \cosh b = \sqrt{\frac{\cosh d+1}{\cosh d-1}}.
\end{equation}
But, from Corollary \ref{cross-coro} (ii), and the fact that $C(f,g)>1$, we have that 
\[
\cosh d = \frac{C(f,g)+1}{C(f,g)-1},
\]
which when substituted to \eqref{eq: equal trace eq3} yields the desired result. 
\end{proof}

Before proceeding to the proof of Theorem \ref{thm: anti classification}, we require some additional results. Let us first mention that an elliptic matrix $A\in\SL$ is \emph{of finite order} if $A^m=I$ for some $m\in\N$, where $I$ is the $2\times2$ identity matrix, and \emph{of infinite order} otherwise. This definition extends to elliptic transformations of $\PSL$ in the obvious way. 

The next result is \cite[Theorem 3]{BBC1996}.

\begin{theorem}[\cite{BBC1996}]\label{thm: beardon dense}
If $f,g\in\PSL$ are non-commuting transformations, and $f$ is elliptic of infinite order, then the semigroup under function composition generated by $\{f,g\}$ is dense in $\PSL$.
\end{theorem}

We also require the following standard result on discrete subgroups of $\PSL$ (see, for example, \cite[Theorem 8.4.1]{Be1995}).

\begin{theorem}\label{thm: non-discrete group}
A non-elementary, non-discrete subgroup of $\PSL$ contains an elliptic transformation of infinite order.
\end{theorem}

We should mention that a subgroup $G$ of $\PSL$ is called \emph{non-elementary} if the orbit $G(z)=\{g(z)\colon g\in G\}$ is infinite for any $z\in\H\cup\overline{\R}$. For our purposes it suffices to note that $G$ is non-elementary whenever it contains two hyperbolic transformations with pairwise distinct fixed points (i.e. with well-defined and non-zero cross-ratio).

Finally, we have a simple result from linear algebra. We provide a short proof for the sake of completeness. 

\begin{lemma}\label{lem: elementary density result}
For any matrices $A,B\in\SL$, the following are equivalent:
\begin{enumerate}
\item The semigroup generated by $\{A,B,-A,-B\}$ is dense in $\SL$.
\item The semigroup generated by $\{A,B\}$ is dense in $\SL$.
\end{enumerate}
\end{lemma}

\begin{proof}
The implication ``(ii) implies (i)" is trivial, since the semigroup generated by $\{A,B\}$ is a subsemigroup of the semigroup generated by $\{A,B,-A,-B\}$.\\
Suppose that the semigroup generated by $\{A,B,-A,-B\}$ is dense in $\SL$. For simplicity, write $S^{\pm}$ for the semigroup generated by $\{A,B,-A,-B\}$, and $S$ for the semigroup generated by $\{A,B\}$. Consider the matrix 
\[
C=\begin{pmatrix}0 & -1\\ 1 &0 \end{pmatrix}\in\SL.
\]
Note that $C^2=-I$, where $I$ is the $2\times 2$ identity matrix. Since $S^{\pm}$ is dense, we can find a sequence $(C_n)\subset S^{\pm}$ that converges to $C$. Observe that either $C_n\in S$ for infinitely many $n\in\N$, or $-C_n\in S$ for infinitely many $n\in\N$. In any case, the sequence $(C_n^2)$ lies in $S$ and converges to $C^2=-I$.\\
Let now $E$ be any matrix in $\SL$, and let $(E_n)\subset S^{\pm}$ be a sequence converging to $E$. If $(E_n)$ lies in $S$ for infinitely many $n\in\N$, we are done. Otherwise, there exists a subsequence $(E_{n_k})$ such that $(-E_{n_k})\subset S$. Then, the sequence $(C_k^2(-E_{n_k}))\subset S$ converges to $-I(-E)=E$.
\end{proof}

We are now ready to prove Theorem \ref{thm: anti classification}. 

\begin{proof}[Proof of Theorem\ref{thm: anti classification}]
Suppose that $A,B\in\SL$ are hyperbolic matrices with $C(A,B)>1$. Moreover, assume that 
\[
\max\left\{\tfrac{1}{2}\lvert \mathrm{tr}(A)\rvert, \tfrac{1}{2}\lvert \mathrm{tr}(B)\rvert\right\}\leq \ \frac{5\ C(A,B)-1}{3\ C(A,B)+1}.
\]
If $f_A,f_B\in\PSL$ are the hyperbolic transformations induced by $A$ and $B$, then Proposition \ref{anti} is applicable to $f_A,f_B$ and yields that $f_A^n\circ f_B^m$ is an elliptic transformation for some $n,m\in\N$. Equivalently, the matrix $E=A^nB^m$ is elliptic, and write $f_E$ for the elliptic transformation induced by $E$; that is, $f_E=f_A^n\circ f_B^m$. We consider two cases, depending on the order of $E$.\\
First, assume that $E$ is of infinite order. We are going to show that the semigroup generated by $\{A,B\}$ is dense in $\SL$. Note that it suffices to show that the semigroup generated by $\{A,E\}$ is dense in $\SL$. Because $C(f_A,f_B)=C(A,B)>1$, the transformations $f_A$ and $f_B$ do not commute, and so neither do $f_A$ and $f_E$. Thus, Theorem \ref{thm: beardon dense} is applicable to $f_A,f_E$ and yields that the semigroup generated by the transformations $\{f_A,f_E\}$ is dense in $\PSL$. This is implies that the semigroup generated by the matrices $\{A,E,-A,-E\}$ is dense in $\SL$, and hence by Lemma \ref{lem: elementary density result} so is the semigroup generated by $\{A,E\}$.\\
Next, we assume that $E$ has finite order; i.e. $E^k=I$, for some $k\in\N$. This means that $A^nB^mA^nB^m\cdots A^nB^m=I$, which immediately implies that $A^{-1},B^{-1}$ lie in the semigroup generated by $\{A,B\}$. Hence, the semigroup generated by $\{A,B\}$ is in fact a group. If this group is discrete, we are done.  Otherwise, the semigroup generated by $\{f_A,f_B\}$ is also a group, which is non-discrete and non-elementary (again because $C(f_A,f_B)>1$). So, using Theorem \ref{thm: non-discrete group} we get that the semigroup generated by $\{f_A,f_B\}$ contains an elliptic transformation of infinite order and we can follow the arguments we presented in the previous paragraph to obtain the desired conclusion. 
\end{proof}

\section{Geometric conditions on pairs}\label{sect: constraints on pairs}

In preparation for the proof of Theorem \ref{MAIN}, that will take place in the next section, we prove certain conditions on pairs of hyperbolic transformations that guarantee domination. This section also contains the proof of Theorem \ref{thm: equal trace}.

Let us first state some additional properties of hyperbolic transformations. Any hyperbolic transformation $f$ can be written as a composition $f=\sigma_1\circ \sigma_2$, where each $\sigma_i$ is the reflection in a hyperbolic geodesic $\gamma_i$ that is perpendicular to the axis $\mathrm{Ax}(f)$, for $i=1,2$. The geodesics $\gamma_1$ and $\gamma_2$ are not unique, but the distance between any such geodesics is fixed and equal to half the translation length; that is
\[
d_\H(\gamma_1,\gamma_2)=\tfrac{1}{2}T_f.
\]
We will say that an interval $I\subset\overline{\mathbb{R}}$ is \emph{symmetric with respect to} $f$ if the unique hyperbolic geodesic with endpoints at the boundary of $I$ is perpendicular to the axis $\mathrm{Ax}(f)$. 

The following observation will prove important in what follows:

\begin{remark}\label{rem: rem}
Let $f$ be a hyperbolic transformation and $M>0$ a constant so that $T_f> M$. Then, we can find intervals $U_f,S_f\subset \overline{\R}$, depending on $M$, with the following properties:
\begin{enumerate}
\item $U_f$ and $S_f$ are symmetric with respect to $f$,
\item $f\left(S_f^c\right)\subset U_f$, where $S_f^c=\overline{\R}\setminus S_f$,
\item if $\gamma_U$ is the hyperbolic geodesic of $\H$ with the same endpoints as $U_f$, and $\gamma_S$ is the hyperbolic geodesic of $\H$ with the same endpoints as $S_f$, then $d_\H(\gamma_U,\gamma_S)=M$.
\end{enumerate}
\end{remark}

\begin{proposition}\label{newlem}
Let $f,g$ be two hyperbolic transformations with $C(f,g)>0$. If
\[
\min\left\{\tfrac{1}{2}\lvert \mathrm{tr}(f)\rvert, \tfrac{1}{2}\lvert \mathrm{tr}(g)\rvert\right\} > \max\left\{\sqrt{C(f,g)},\frac{1}{\sqrt{C(f,g)}}\right\},
\]
then there exist open intervals $U_f, S_f,U_g,S_g$ in $\overline{\mathbb{R}}$, with pairwise disjoint closures, that satisfy the following properties:
\begin{enumerate}
\item The intervals $U_f,S_f$ are symmetric with respect to $f$ and $f(S_f^c)$ is contained in $U_f$; and,
\item the intervals $U_g,S_g$ are symmetric with respect to $g$ and $g(S_g^c)$ is contained in $U_g$.
\end{enumerate}
In particular, $\{f,g\}$ is dominated.
\end{proposition}

\begin{proof}
Let $f,g$ be hyperbolic transformations as in the statement of the theorem. Observe that if we prove the existence of the intervals $U_f,S_f,U_g,S_g$, then  Theorem \ref{aby} immediately tells us that $\{f,g\}$ is dominated. Since $C(f,g)>0$ the axes of $f$ and $g$ are disjoint and a hyperbolic distance $d>0$ apart. We first consider the case where $C(f,g)>1$. By assumption we have that
\begin{equation}\label{eq: newlem proof eq1}
\min\left\{\tfrac{1}{2}\lvert \mathrm{tr}(f) \rvert, \tfrac{1}{2}\lvert \mathrm{tr}(g) \rvert \right\}> \sqrt{C(f,g)}=\max\left\{\sqrt{C(f,g)},\frac{1}{\sqrt{C(f,g)}}\right\}.
\end{equation}
Equation \eqref{eq: trace and translation length} tells us that $\tfrac{1}{2}\lvert \tr(f)\rvert =\cosh \tfrac{1}{2}T_f$, and similarly for $g$. Thus, \eqref{eq: newlem proof eq1} is equivalent to
\[
\min\left\{\cosh\tfrac{1}{2}T_f, \cosh\tfrac{1}{2}T_g \right\}> \sqrt{C(f,g)},
\]
which in turn can be rewritten as 
\begin{equation}\label{eq: newlem proof eq2}
\min\left\{\sinh\tfrac{1}{2}T_f, \sinh\tfrac{1}{2}T_g \right\}> \sqrt{C(f,g)-1}.
\end{equation}
Recall that Lemma~\ref{crl} tell us that $C(f,g)=\coth^2\tfrac{1}{2}d$, meaning that $\frac{1}{\sinh^2\tfrac{1}{2}d}=C(f,g)-1$. We are thus led to 
\begin{equation}\label{eq: newlem proof eq3}
\min\left\{\sinh\tfrac{1}{2}T_f, \sinh\tfrac{1}{2}T_g \right\}> \frac{1}{\sinh \tfrac{1}{2}d}.
\end{equation}
Now, let $\ell$ be the unique hyperbolic geodesic that is perpendicular to $\mathrm{Ax}(f)$ and $\mathrm{Ax}(g)$, and let $\sigma$ the reflection in $\ell$. Also, define the reflections $\sigma_f=f\circ \sigma$ and $\sigma_g=\sigma\circ  g$ and let $\ell_f, \ell_g$ be their geodesics of reflection respectively. Observe that $\ell_f$ is also perpendicular to $\mathrm{Ax}(f)$ and $\ell_g$ is perpendicular to $\mathrm{Ax}(g)$ (see Figure \ref{fig-anti}).\\
To fully justify the configuration shown in Figure \ref{fig-anti}, we first claim that the geodesic $\ell_f$ does not intersect $\mathrm{Ax}(g)$ in $\overline{\H}$ and $\ell_g$ does not intersect $\mathrm{Ax}(f)$. Assuming that the geodesic $\ell_f$ meets $\mathrm{Ax}(g)$ in $\overline{\mathbb{H}}$ at an angle $\phi\in [0,\pi/2)$ we can see that the geodesics $\ell_f,\mathrm{Ax}(f),\ell$ and $\mathrm{Ax}(g)$ form a quadrilateral in $\mathbb{H}$ that has three right angles and the fourth is $\phi$. Moreover, $\phi$ cannot be $\pi/2$ because then $\ell_f$ would coincide with $\ell$. Applying \cite[Theorem 7.17.1 (i)]{Be1995} to this quadrilateral implies that 
\[
\sinh\tfrac{1}{2}T_f = \sinh d_\H(\ell,\ell_f)=\frac{\cos\phi}{\sinh d }\leq \frac{1}{\sinh\tfrac{1}{2} d},
\]
which directly contradicts \eqref{eq: newlem proof eq3}. We reach the same contradiction if we assume that $\ell_g$ intersects $\mathrm{Ax}(f)$ in $\overline{\H}$.\\
Next, we show that the geodesics $\ell_f,\ell_g$ do not intersect in $\overline{\H}$. Working towards a contradiction we suppose that  $\ell_f$ and $\ell_g$ meet in $\overline{\mathbb{H}}$ at an angle $\theta \in[0,\pi/2]$. Then, because the pairs $\ell_f,\mathrm{Ax}(g)$ and $\ell_g,\mathrm{Ax}(f)$ are disjoint, we get that the geodesics $\ell,\ell_f,\ell_g,\mathrm{Ax}(f), \mathrm{Ax}(g)$ define a pentagon in $\overline{\mathbb{H}}$ with four right angles and the fifth being $\theta$. Using \cite[Theorem 7.18.1]{Be1995} we get
\begin{equation}\label{eq: newlem proof eq4}
\sinh\tfrac{1}{2}T_f\sinh\tfrac{1}{2}T_g\cosh d - \cosh\tfrac{1}{2}T_f\cosh\tfrac{1}{2}T_g = \cos\theta\leq1.
\end{equation}
Consider the function $h\colon \mathbb{R}^2\to \mathbb{R}$, with $h(x,y)=\sinh x\sinh y\cosh d-\cosh x\cosh y$, and let $b\in\mathbb{R}$ be such that
\[
\sinh b=\frac{1}{\sinh \tfrac{1}{2}d}=\sqrt{\frac{2}{\cosh d-1}}.
\]
The function $h$ and the number $b$ also appeared in the proof of Proposition~\ref{anti}, where we showed that $h(x,b)$ is increasing for all $x>0$ (see the arguments succeeding \eqref{increasing}), and $h(b,b)=1$ (see \eqref{eq: antiparallel hbb}). Also, $h(x,x)=\cosh d \sinh^2x -\cosh^2x$ is an increasing function of $x$. Combining these facts we obtain that $h(x,b),h(x,x)>1$, for all $x>b$. Let us fix $x_0>b$. Now, again in the proof of Proposition~\ref{anti} we showed that $h(x_0,y)$ is increasing for all $y\in(0,x_0)$ (see the arguments succeeding \eqref{eq: antiparallel final der}). We thus have that
\[
h(x_0,y)>h(x_0,b)>1,\quad\text{for all}\ y\in (b,x_0).
\]
Since $h(x,y)=h(y,x)$, we can conclude that all points $(x,y)\in\mathbb{R}^2$ with $\min\{x,y\}>b$ satisfy $h(x,y)>1$. Now note that by \eqref{eq: newlem proof eq3} we have that $\tfrac{1}{2}T_f,\tfrac{1}{2}T_g>b$, meaning that $h\left(\tfrac{1}{2}T_f,\tfrac{1}{2}T_f\right)>1$, contradicting \eqref{eq: newlem proof eq4}.\\
To sum up, we showed that $\ell_f$ is disjoint with $\mathrm{Ax}(g)$ and $\ell_g$, and similarly $\ell_g$ is disjoint with $\mathrm{Ax}(f)$ and $\ell_f$. Hence, if we denote by $U_f$ the open interval in $\overline{\mathbb{R}}$ that has the same endpoints as $\ell_f$ and contains $u(f)$, and by $I_g$ the open interval that has the same endpoints as $\ell_g$ and contains $s(g)$, we obtain that $\overline{U_f}\cap \overline{I_g}=\emptyset$. Also, let $I_f=\sigma(U_f)$ and $U_g=\sigma(I_g)$. Since $f=\sigma_f\circ\sigma$ and $g = \sigma\circ \sigma_g$, we have that $f(I_f^c)= \overline{U_f}$ and $g(I_g^c)= \overline{U_g}$. It is now easy to see that we can choose slightly larger intervals $I_f\subset S_f$ and $I_g\subset S_g$ so that $U_f,S_f,U_g,S_g$ all have disjoint closures, $f(S_f^c)\subset U_f$ and $g(S_g^c)\subset U_g$, as required. This concludes our proof for the case $C(f,g)>1$.\\
Recall that because $f,g$ are hyperbolic, we have that $C(f,g)\neq1$. But, for the scenario $0<C(f,g)<1$, we can obtain the desired intervals by applying the first case to the pair $\{f^{-1},g\}$, because $C(f^{-1},g)=1/C(f,g)>1$.
\end{proof}

\begin{figure}[ht]
\centering
\begin{tikzpicture}

\begin{scope}[scale = 2.5,decoration={
    markings,
    mark=at position 0.39 with {\arrow{To}}}
    ] 
\draw [line width=0.4pt] (0.,0.) circle (1.cm);
\draw [shift={(-1.8312065528259258,0.)},line width=0.4pt, postaction={decorate}]  plot[domain=0.5776873665141405:-0.5776873665141409,variable=\t]({1.*1.5340526194080204*cos(\t r)+0.*1.5340526194080204*sin(\t r)},{0.*1.5340526194080204*cos(\t r)+1.*1.5340526194080204*sin(\t r)});

\draw [shift={(-0.5460880414919854,-0.9365883317794479)},line width=0.4pt]  plot[domain=-0.13124699917441784:2.2170973274583603,variable=\t]({1.*0.4188195939613624*cos(\t r)+0.*0.4188195939613624*sin(\t r)},{0.*0.4188195939613624*cos(\t r)+1.*0.4188195939613624*sin(\t r)});
\draw [shift={(0.5460880414919863,-0.9365883317794478)},line width=0.4pt]  plot[domain=0.9244953261314339:3.2728396527642105,variable=\t]({1.*0.41881959396136353*cos(\t r)+0.*0.41881959396136353*sin(\t r)},{0.*0.41881959396136353*cos(\t r)+1.*0.41881959396136353*sin(\t r)});

\draw[line width=.4pt] (-1,0) -- (1,0);
\end{scope}

\begin{scope}[scale = 2.5,decoration={
    markings,
    mark=at position 0.635 with {\arrow{To}}}
    ] 
\draw [shift={(1.8312065528259265,0.)},line width=0.4pt, postaction={decorate}]  plot[domain=3.719280020103933:2.563905287075653,variable=\t]({1.*1.5340526194080215*cos(\t r)+0.*1.5340526194080215*sin(\t r)},{0.*1.5340526194080215*cos(\t r)+1.*1.5340526194080215*sin(\t r)});

\node [above] at (-.7,0) {$\ell$};
\node [right] at (-.3,.2) {$f$};
\node [left] at (.3,.2) {$g$};
\node [above] at (-.7, -.55) {$\ell_f$};
\node [above] at (.7,-.55) {$\ell_g$};
\end{scope}
\end{tikzpicture}
\caption{The hyperbolic transformations in Proposition \ref{newlem}.}
\label{fig-anti} 
\end{figure}

Using Proposition \ref{newlem} and Lemma \ref{lem: equal trace lemma} we can now prove Theorem \ref{thm: equal trace}.

\begin{proof}[Proof of Theorem \ref{thm: equal trace}]
We assume that $f,g$ are hyperbolic transformations with $C(f,g)>1$ and $\lvert \tr(f)\rvert=\lvert \tr(g)\rvert=t$. Our goal is to show that $\{f,g\}$ is dominated if and only if $\tfrac{1}{2}t>\sqrt{C(f,g)}$.\\
By Lemma \ref{lem: equal trace lemma} we have that if $\tfrac{1}{2} t\leq \sqrt{C(f,g)}$ then $f\circ g$ is either elliptic or parabolic. Since the semigroup generated by a dominated set only contains hyperbolic transformations, $\{f,g\}$ is not dominated in this case. On the other hand, if $\tfrac{1}{2} t> \sqrt{C(f,g)}$, Proposition \ref{newlem} immediately yields that $\{f,g\}$ is dominated, since
\[
\max\left\{\sqrt{C(f,g)},\frac{1}{\sqrt{C(f,g)}}\right\} = \sqrt{C(f,g)}.\qedhere
\]
\end{proof}

Next, we prove the analogue of Proposition \ref{newlem} for two hyperbolic transformations with crossing axes.

\begin{proposition}\label{inter-thm}
Suppose that $f$ and $g$ are hyperbolic transformations with $C(f,g)<0$. If 
\[
\min\left\{\tfrac{1}{2}\lvert\tr(f)\rvert, \tfrac{1}{2}\lvert\tr(g)\rvert\right\}>\max\left\{\sqrt{1-C(f,g)},\sqrt{1-\frac{1}{C(f,g)}}\right\},
\]
then there exist open intervals $U_f, S_f,U_g,S_g$ in $\overline{\mathbb{R}}$, with pairwise disjoint closures, that satisfy the following properties:\begin{enumerate}
\item The intervals $U_f,S_f$ are symmetric with respect to $f$ and $f(S_f^c)$ is contained in $U_f$; and,
\item the intervals $U_g,S_g$ are symmetric with respect to $g$ and $g(S_g^c)$ is contained in $U_g$.
\end{enumerate}
\end{proposition}

\begin{proof}
Since $C(f,g)<0$, the axes of $f$ and $g$ meet at an angle $\theta\in(0, \pi)$. For convenience, we conjugate $f$ and $g$ with a transformation in $\mathrm{PSL}(2,\C)$, so that they act on the unit disc $\D$, their axes cross at the origin and the Euclidean diameter joining $i$ and $-i$ bisects $\theta$. By the symmetry of this configuration, we also have that the Euclidean diameter joining $1$ and $-1$ bisects $\pi-\theta$. Throughout the proof, we use the notation $[z,w]$ to denote the \emph{closed} arc of the unit circle running counter-clockwise from $z\in\partial\mathbb{D}$ to $w\in\partial\D$. Similarly, $(z,w)$ will denote the respective \emph{open} arc. By applying a Euclidean rotation about the origin, we can also assume that both $f$ and $g$ map the arc $[-1,1]$ inside itself, and without loss of generality assume that the arc $[u(f),u(g)]$ does not contain $s(f)$ and $s(g)$. See Figure \ref{intersect} for the geometry of the configuration of $f$ and $g$.\\
We first consider the case where $\theta\in(0, \tfrac{\pi}{2}]$. Due to Corollary \ref{cross-coro} this is equivalent to $C(f,g)\in[-1,0)$, meaning that in this case $\sqrt{1-\frac{1}{C(f,g)}}\geq \sqrt{1-C(f,g)}$, and so
\begin{equation}\label{eq: intersecting eq0}
\max\left\{\sqrt{1-C(f,g)},\sqrt{1-\frac{1}{C(f,g)}}\right\}=\sqrt{1-\frac{1}{C(f,g)}}.
\end{equation}
We shall work independently for $f$ and $g$. Starting with $f$, let $\ell_f$ be the hyperbolic geodesic of $\D$ landing at $-i$ that is perpendicular to $\mathrm{Ax}(f)$, the axis of $f$. Also, write $z_f\in\partial\D$ for the other endpoint of $\ell_f$. Since $\theta\leq\tfrac{\pi}{2}$ and the diameter joining $-i$ and $i$ bisects $\theta$, we have that $z_f$ lies in $[-1,-i]$. Let $a_f\in\D$ be the point of intersection of $\ell_f$ and $\mathrm{Ax}(f)$. By the symmetry of our configuration, we have that the hyperbolic geodesic $\widetilde{\ell_f}$ with endpoints $i$ and $-z_f$ is also perpendicular to $\mathrm{Ax}(f)$ and $-z_f$ lies in $[1,i]$. Also, the point of intersection of $\widetilde{\ell_f}$ and $\mathrm{Ax}(f)$ is $-a_f$. By Remark \ref{rem: rem} we have that if $T_f>d_\D(a_f,-a_f)$, then there exist open arcs $U_f\subset(z_f,-i)$ and $S_f\subset (-z_f,i)$, that are symmetric with respect to $f$, so that $f(S_f^c)\subset U_f$.\\
Our goal now is to show that the condition $T_f>d_\D(a_f,-a_f)$ is equivalent to $\tfrac{1}{2}\lvert\tr(f)\rvert>\sqrt{1-\frac{1}{C(f,g)}}$. Observe that we have $d_\D(a_f,-a_f)=d_\D(a_f,0)+d_\D(0,-a_f)$ and $d_\D(a_f,0)=d_\D(0,-a_f)$. So the condition $T_f>d_\D(a_f,-a_f)$ is equivalent to $\tfrac{1}{2}T_f>d_\D(a_f,0)$. Consider the hyperbolic triangle with vertices $a_f,0$ and $-i$. This triangle has angles $0,\tfrac{\pi}{2}$ and $\tfrac{\theta}{2}$, meaning that we can use the Angle of Parallelism \cite[Theorem 7.9.1]{Be1995} in order to obtain that 
\begin{equation}\label{eq: intersecting eq1}
\cosh d_\D(a_f,0)\ \sin\tfrac{\theta}{2}=1.
\end{equation}
Equation \eqref{eq: intersecting eq1} can be rewritten as
\begin{equation}\label{eq: intersecting eq2}
\cosh d_\D(a_f,0) = \frac{1}{\sqrt{1-\cos^2\tfrac{\theta}{2}}}=\sqrt{1-\frac{1}{C(f,g)}},
\end{equation}
where for the last equality we used Corollary \ref{cross-coro} (i). Hence, using \eqref{eq: intersecting eq2} and the equation $\cosh\tfrac{1}{2}T_f=\tfrac{1}{2}\lvert\tr(f)\rvert$, i.e. \eqref{eq: trace and translation length}, we have that
\begin{align*}
\tfrac{1}{2}T_f>d_\D(a_f,0)& & \iff & &\tfrac{1}{2}\lvert\tr(f)\rvert=\cosh\tfrac{1}{2}T_f> \cosh d_\D(a_f,0) = \sqrt{1-\frac{1}{C(f,g)}}.
\end{align*}
Following the same arguments for $g$ in place of $f$, we can see that if $\tfrac{1}{2}\lvert\tr(g)\rvert>\sqrt{1-\frac{1}{C(f,g)}}$, then there exists a point $z_g\in[-i,1]$ and arcs $U_g\subset(-i,z_g)$ and $S_g\subset(i,-z_g)$, symmetric with respect to $g$, and such that $g(S_g^c)\subset U_g$. By construction, we have that the arcs $U_f,S_f,U_g,S_g$ have pairwise disjoint closures, which concludes the proof in the case where $\theta\in(0,\tfrac{\pi}{2}]$.\\
For the case $\theta\in(\tfrac{\pi}{2},\pi)$, we can apply the above arguments to the pair $f^{-1},g$, since their axes cross at angle $\pi-\theta\in(0,\tfrac{\pi}{2})$. Recalling that $C(f^{-1},g)=1/C(f,g)$ yields the desired conclusion.
\end{proof}

\begin{figure}[ht]
\centering
\begin{tikzpicture}[scale=3,decoration={
    markings,
    mark=at position 0.8 with {\arrow{To}}}
    ] 
\coordinate (o) at (0,0);
\coordinate (l) at (-0.71, -0.71);
\coordinate (r) at (0, -1);
\coordinate (tr) at (-0.71,0.71);
\coordinate (tl) at (0.71,0.71);

\draw [line width=0.4pt] (0.,0.) circle (1.cm);
\draw [line width=0.4pt, postaction= {decorate}] (0.7071067811865475,0.7071067811865475) -- (-0.7071067811865476,-0.7071067811865476);
\draw [line width=0.4pt, postaction= {decorate}] (-0.7071067811865477,0.7071067811865475) --  (0.7071067811865476,-0.7071067811865474);
\draw [dashed, line width=0.4pt] (0,-1)-- (0,0);
\pic [draw, "$\tfrac{\theta}{2}$", angle eccentricity=1.5, line width=0.4pt] {angle = l--o--r};
\pic [draw, "$\theta$", angle eccentricity=1.5, line width=0.4pt] {angle = tl--o--tr};
\filldraw[black] (-0.71,-0.71) circle (.5pt) node[below left] {$u(f)$};
\filldraw[black] (0.71,-0.71) circle (.5pt) node[below right] {$u(g)$};
\filldraw[black] (0.71,0.71) circle (.5pt) node[above right] {$s(f)$};
\filldraw[black] (-0.71,0.71) circle (.5pt) node[above left] {$s(g)$};
\filldraw[black] (0,-1) circle (.5pt) node[below] {$-i$};
\filldraw[black] (-1,0) circle (.5pt) node[left] {$z_f$};
\filldraw[black] (1,0) circle (.5pt) node[right] {$-z_f$};
\filldraw[black] (-0.29,-0.29) circle (.5pt) node[above, yshift=2pt] {$a_f$};
\filldraw[black] (0.29,0.29) circle (.5pt) node[right] {$-a_f$};
\filldraw[black] (0,0) circle (.5pt) node[right, xshift=2pt] {0};
\filldraw[black] (0,1) circle (.5pt) node[above] {$i$};
\node [left] at (-0.40,-0.37) {$f$};
\node [right] at (0.41,-0.38) {$g$};

\draw [shift={(-1.,-1.)}]  plot[domain=0.:1.5707963267948968,variable=\t]({1.*1.*cos(\t r)+0.*1.*sin(\t r)},{0.*1.*cos(\t r)+1.*1.*sin(\t r)});
\draw [shift={(1.,1.)}]  plot[domain=0.:1.5707963267948968,variable=\t]({0.*1.*cos(\t r)+-1.*1.*sin(\t r)},{-1.*1.*cos(\t r)+0.*1.*sin(\t r)});

\end{tikzpicture}
\caption{Hyperbolic transformations with crossing axes in Proposition~\ref{inter-thm}.}\label{intersect}
\end{figure}

The next theorem combines the cases presented in Propositions \ref{newlem} and \ref{inter-thm} into a single statement. 

\begin{theorem}\label{thm: two-gen constraints}
Suppose that $f,g$ are hyperbolic transformations with $C(f,g)\neq0$. If
\[
\min\left\{\tfrac{1}{2}\lvert\tr(f)\rvert, \tfrac{1}{2}\lvert\tr(g)\rvert\right\}>\max\left\{\sqrt{1+\left\lvert C(f,g)\right\rvert},\sqrt{1+\frac{1}{\left\lvert C(f,g)\right\rvert}}\right\},
\]
then there exist open intervals $U_f, S_f,U_g,S_g$ in $\overline{\mathbb{R}}$, with pairwise disjoint closures, that satisfy the following properties: \begin{enumerate}
\item The intervals $U_f,S_f$ are symmetric with respect to $f$ and $f(S_f^c)$ is contained in $U_f$; and,
\item the intervals $U_g,S_g$ are symmetric with respect to $g$ and $g(S_g^c)$ is contained in $U_g$.
\end{enumerate}
\end{theorem}

\begin{proof}
If $C(f,g)<0$, then the the statement of the theorem is identical to the statement of Proposition \ref{inter-thm}. If $C(f,g)>0$, then 
\begin{align*}
\min\left\{\tfrac{1}{2}\lvert\tr(f)\rvert, \tfrac{1}{2}\lvert\tr(g)\rvert\right\}&>\max\left\{\sqrt{1+\left\lvert C(f,g)\right\rvert},\sqrt{1+\frac{1}{\left\lvert C(f,g)\right\rvert}}\right\}\\
&>\max\left\{\sqrt{C(f,g)},\sqrt{\frac{1}{C(f,g)}}\right\},
\end{align*}
and so the conclusion is obtained by applying Proposition \ref{newlem}.
\end{proof}

We end this section with a simple remark combining Remark \ref{rem: rem} and Theorem \ref{thm: two-gen constraints}.

\begin{remark}\label{rem: choice of intervals}
Let $f$ and $g$ by hyperbolic transformations with $C(f,g)\neq 0$, satisfying
\[
\min\left\{\tfrac{1}{2}\lvert\tr(f)\rvert, \tfrac{1}{2}\lvert\tr(g)\rvert\right\}>\max\left\{\sqrt{1+\left\lvert C(f,g)\right\rvert},\sqrt{1+\frac{1}{\left\lvert C(f,g)\right\rvert}}\right\},
\]
as in the assumption of Theorem \ref{thm: two-gen constraints}. For any $\epsilon>0$ with
\[
\tfrac{1}{2}\lvert\tr(f)\rvert> \max\left\{\sqrt{1+\left\lvert C(f,g)\right\rvert},\sqrt{1+\frac{1}{\left\lvert C(f,g)\right\rvert}}\right\}\cdot\epsilon,
\]
the intervals $U_f,S_f$ that appear in the conclusion of Theorem \ref{thm: two-gen constraints} can be chosen to have the following property:\\
If $\ell_U,\gamma_S$ are the hyperbolic geodesics of $\H$ with the same endpoints as $U_f$ and $S_f$, respectively, then $d_\H(\ell_U,\gamma_S)$, the hyperbolic distance between $\ell_U$ and $\gamma_S$, satisfies 
\[
\max\left\{\sqrt{1+\left\lvert C(f,g)\right\rvert},\sqrt{1+\frac{1}{\left\lvert C(f,g)\right\rvert}}\right\}\cdot \epsilon=\cosh \tfrac{1}{2}d_\H(\ell_U,\gamma_S).
\]
An analogous property will hold for the intervals $U_g,S_g$.
\end{remark}

\section{Proof of Theorem~\ref{MAIN}} \label{sect: proof of main theorem}
In this section we prove Theorem \ref{MAIN}, which is restated below in the setting of $\PSL$, for convenience. 

\begin{theorem*}
Let $u_1,u_2,\dots,u_n,s_1,s_2,\dots s_N\in\R\P^1$ be $2N$ pairwise distinct points in the projective real line. Suppose that $f_1,f_2,\dots f_N\in\PSL$ are  hyperbolic transformations with $u(f_i)=u_i$ and $s(f_i)=s_i$, for $i=1,2,\dots,N$, and consider the constants 
\[
M_{i,j} = \max\left\{1+\left\lvert C(f_i,f_j)\right\rvert,1+\frac{1}{\left\lvert C(f_i,f_j)\right\rvert}\right\},
\]
for all pairs $i,j\in\{1,2,\dots,N\}$ with $i\neq j$. If 
\begin{equation}\label{eq: main assumption}
\tfrac{1}{2}\lvert \tr(f_k)\rvert >4\ \max_{i\neq j}\{M_{i,j}\}\cdot \max_{i\neq j}\left\{\frac{\lvert C(f_i,f_j)\rvert +1}{\lvert C(f_i,f_j)-1\rvert}\right\}, \quad \text{for all}\ k=1,2,\dots,N,
\end{equation}
then the set $\{f_1,f_2,\dots, f_N\}$ is dominated. 
\end{theorem*}

Before proceeding with the proof of Theorem \ref{MAIN}, let us briefly discuss Schottky groups. We say that the group $G$ generated by a set $\{f_1,f_2,\dots,f_N\}\subset\PSL$ is a \emph{(classical) Schottky group}, if there exist $2N$ open intervals $U^1,S^1,U^2,S^2,\cdots,U^N,S^N\subset \overline{\R}$ with pairwise disjoint closures, so that $f_i(\overline{\R}\setminus S^i)\subset U^i$, for all $i=1,2,\dots, N$. Due to a result of Maskit \cite{Ma1967}, $G$ is a Schottky group if and only if it is a free, discrete subgroup of $\PSL$, all of whose non-identity elements are hyperbolic.

Note that if the group generated by $\{f_1,f_2,\dots,f_N\}$ is a Schottky group, and $U^i,S^i$ are intervals of $\overline{\R}$ as in the previous paragraph, for $i=1,2,\dots,N$, then each $f_i$ maps $\overline{\bigcup_{k=1}^N U^k}$ inside the open interval $U^i$. Therefore, Theorem \ref{aby} implies that $\{f_1,f_2,\dots, f_N\}$ is dominated. The converse, however, does not hold, as one can see for a pair $\{f,g\}\subset\PSL$ of hyperbolic transformations with $u(f)=u(g)$, which is dominated but does not generate a discrete group.

\begin{proof}[Proof of Theorem \ref{MAIN}]
Suppose that $f_1,f_2,\dots,f_N$ are hyperbolic transformations, with the properties described in the statement of the theorem. Note that, by assumption, the fixed points of all $f_i$ are distinct, meaning that $C(f_i,f_j)$ is well-defined and $C(f_i,f_j)\neq0$, for all $i\neq j$. Due to the discussion preceding the proof, it suffices to show that under the trace condition in \eqref{eq: main assumption} we can find open intervals $U^1,S^1,U^2,S^2\dots, U^k,S^k\subset\overline{\R}$, with pairwise disjoint closures, so that each $f_k$ maps $\overline{\R}\setminus S^k$ into $U^k$. This would also prove that the group generated by $\{f_1,f_2,\dots, f_N\}$ is Schottky, as we mentioned in the Introduction.\\
Let $\epsilon>1$ be such that for all $k=1,2,\dots,N$ we have that
\begin{equation}\label{eq: proof epsilon}
\tfrac{1}{2}\lvert\tr(f_k)\rvert > 4\epsilon\ \max_{i\neq j}\{M_{i,j}\}\cdot \max_{i\neq j}\left\{\frac{\lvert C(f_i,f_j)\rvert +1}{\lvert C(f_i,f_j)-1\rvert}\right\}.
\end{equation}
Fix $k=1,2,\dots,N$. For any $i\neq k$, the hypothesis \eqref{eq: main assumption} and \eqref{eq: proof epsilon} yield  
\begin{align}
\min\{\tfrac{1}{2}\lvert \tr(f_k)\rvert,\tfrac{1}{2}\lvert \tr(f_i)\rvert\}&> 4\epsilon\cdot M_{k,i}\cdot \frac{\lvert C(f_k,f_i)\rvert +1}{\lvert C(f_k,f_i)-1\rvert} \nonumber\\
	&=4\epsilon\cdot \max\left\{1+\left\lvert C(f_k,f_i)\right\rvert,1+\frac{1}{\left\lvert C(f_k,f_i)\right\rvert}\right\}\cdot \frac{\lvert C(f_k,f_i)\rvert +1}{\lvert C(f_k,f_i)-1\rvert} \nonumber\\
	&>\sqrt{\epsilon} \cdot \max\left\{\sqrt{1+\left\lvert C(f_k,f_i)\right\rvert},\sqrt{1+\frac{1}{\left\lvert C(f_k,f_i)\right\rvert}}\right\} \label{eq: proof estimate 2} \\
	&> \max\left\{\sqrt{1+\left\lvert C(f_k,f_i)\right\rvert},\sqrt{1+\frac{1}{\left\lvert C(f_k,f_i)\right\rvert}}\right\}\nonumber ,
\end{align}
which means that Theorem \ref{thm: two-gen constraints} is applicable to the pair of hyperbolic transformations $f_k,f_i$. So we can find intervals $U_i^k,U_k^i,S_i^k,S_k^i\subset \overline{\R}$ with pairwise disjoint closures so that $U_i^k,S_i^k$ are symmetric with respect to $\mathrm{Ax}(f_k)$, the axis of $f_k$, and $U_k^i,S_k^i$ are symmetric with respect to $\mathrm{Ax}(f_i)$ . Also the intervals satisfy $f_k\left(\overline{\R}\setminus S_i^k\right)\subset U_i^k$ and $f_i\left(\overline{\R}\setminus S_k^i\right)\subset U_k^i$.\\
Moreover, in \eqref{eq: proof estimate 2} we got that
\[
\min\{\tfrac{1}{2}\lvert \tr(f_k)\rvert,\tfrac{1}{2}\lvert \tr(f_i)\rvert\}> \sqrt{\epsilon}\cdot \max\left\{\sqrt{1+\left\lvert C(f_k,f_i)\right\rvert},\sqrt{1+\frac{1}{\left\lvert C(f_k,f_i)\right\rvert}}\right\}.
\]
Hence, according to Remark \ref{rem: choice of intervals}, we can choose the intervals $U_i^k,U_k^i,S_i^k,S_k^i$ to have the following properties:\\
Let $\ell_i^k,\ell_k^i$ be the hyperbolic geodesics of $\H$ with the same endpoints as $U_i^k$ and $U_k^i$, respectively. Similarly, $\gamma_i^k,\gamma_k^i$ are the hyperbolic geodesics of $\H$ with the same endpoints as $S_i^k$ and $S_k^i$, respectively. Then the geodesics $\ell_i^k,\ell_k^i,\gamma_i^k,\gamma_k^i$ are pairwise distinct and satisfy
\begin{equation}\label{eq: proof distance of intervals}
\cosh\tfrac{1}{2}d_\H(\ell_i^k,\gamma_i^k)= \cosh\tfrac{1}{2}d_\H(\ell_k^i,\gamma_k^i)=\max\left\{\sqrt{1+\left\lvert C(f_k,f_i)\right\rvert},\sqrt{1+\frac{1}{\left\lvert C(f_k,f_i)\right\rvert}}\right\}\cdot \sqrt{\epsilon}.
\end{equation}
By repeating this process for all $i=1,2,\dots,N$, with $i\neq k$ we obtain two collections of $N-1$ intervals $\mathbb{U}^k=\{U_1^k,\dots, U^k_{k-1},U^k_{k+1},\dots, U_N^k\}$ and $\mathbb{S}^k=\{S_1^k,\dots, S^k_{k-1},S^k_{k+1},\dots, S_N^k\}$. Since $U^k_i,S^k_i$ are symmetric with respect to the axis of $f_k$, the collections $\mathbb{U}^k$ and $\mathbb{S}^k$ are totally ordered with the set inclusion operation; i.e. for any $U^k_i,U^k_j\in\mathbb{U}^k$ we have that either $U^k_i\subseteq U^k_j$ or $U^k_j\subseteq U^k_i$, and similarly for the intervals in $\mathbb{S}^k$. Let $U^k_n$ be a least element of $\mathbb{U}^k$ (an inner-most interval in the collection) and $S^k_m$ a least element of $\mathbb{S}^k$. Here, we say \emph{a} least element, and not \emph{the}, since they are not necessarily unique. Furthermore, the indices $n$ and $m$ might not coincide. Let us also consider the intervals $U^k_m$ and $S^k_n$. Because $U^k_n$ and $S^k_m$ are least elements, we have that $U^k_n\subseteq U^k_m$ and $S^k_m\subseteq S^k_n$. Now, observe that 
\[
\overline{U^k_n}\cap \overline{S^k_m}\subset \overline{U^k_n}\cap \overline{S^k_n}=\emptyset,
\]
because, by construction, $U^k_n$ and $S^k_n$ have disjoint closures.\\
Now let $k_0=1,2,\dots,N$ with $k\neq k_0$. Following the process described above for $k_0$ in place of $k$, we obtain collections $\mathbb{U}^{k_0}, \mathbb{S}^{k_0}$ of intervals that are symmetric with respect to $\mathrm{Ax}(f_{k_0})$. Let $U^{k_0}_{n_0}$ and $S^{k_0}_{m_0}$ be some least elements in $\mathbb{U}^{k_0}$ and $\mathbb{S}^{k_0}$, respectively. We want to also show that the intervals $U^k_n,S^k_m,U^{k_0}_{n_0},S^{k_0}_{m_0}$ have pairwise disjoint closures, where $U^k_n,S^k_m$ are the intervals that were constructed above. First, observe  that as we showed before, we have that $\overline{U^{k_0}_{n_0}}\cap\overline{S^{k_0}_{m_0}}=\emptyset$. Since $U^k_n,U^{k_0}_{n_0}$ are least elements in their respective collections, we have that 
\[
\overline{U^{k_0}_{n_0}}\cap\overline{U^k_n}\subset\overline{U^{k_0}_k}\cap\overline{U^k_{k_0}}=\emptyset,
\]
because by construction we have that $U^k_{k_0},S^k_{k_0},U^{k_0}_k,S^{k_0}_k$ have pairwise disjoint closures. Similar arguments yield the desired conclusion.\\
The process we just described allows us to construct intervals $U^1,U^2,\dots,U^N, S^1,S^2,\dots,S^N$ (we suppress the indices and keep the exponents, for visual clarity) that have pairwise disjoint closures, and so that each pair $U^i,S^i$ is symmetric with respect to the axis $\mathrm{Ax}(f_i)$ of $f_i$, for $i=1,2,\dots,N$. Therefore, our proof will be complete upon showing that $f_i\left(\overline{\R}\setminus S^i\right)\subset U^i$, for all $i=1,2,\dots,N$.\\
Focusing on the exponent $k$ we fixed earlier, and reintroducing the indices in our intervals, that is considering the inner-most intervals $U^k_n$ and $S^k_m$ described above, we have to show that
\begin{equation}\label{eq: end goal}
f_k\left(\overline{\R}\setminus S^k_m\right)\subset U^k_n.
\end{equation}
First, observe that if $n=m$, then we already have the desired conclusion, by construction. So, for the rest of the proof, we assume that $n\neq m$. We will also use the intervals $U^k_m$ and $S^k_ n$. Let $\ell^k_n,\ell^k_m$ be the geodesics of $\H$ with the same endpoints as $U^k_n,U^k_m$, respectively. Similarly, $\gamma^k_n,\gamma^k_m$ be the geodesics of $\H$ with the same endpoints as $S^k_n,S^k_m$, respectively. Because the intervals in our construction satisfied \eqref{eq: proof distance of intervals}, we have that
\begin{equation}\label{eq: distance of intervals n}
\cosh\tfrac{1}{2}d_\H(\ell_n^k,\gamma_n^k)=\max\left\{\sqrt{1+\left\lvert C(f_k,f_n)\right\rvert},\sqrt{1+\frac{1}{\left\lvert C(f_k,f_n)\right\rvert}}\right\}\cdot \sqrt{\epsilon},
\end{equation}
and 
\begin{equation}\label{eq: distance of intervals m}
\cosh\tfrac{1}{2}d_\H(\ell_m^k,\gamma_m^k)=\max\left\{\sqrt{1+\left\lvert C(f_k,f_m)\right\rvert},\sqrt{1+\frac{1}{\left\lvert C(f_k,f_m)\right\rvert}}\right\}\cdot \sqrt{\epsilon}.
\end{equation}
Observe that if $T_{f_k}$, the translation length of $f_k$, satisfies $T_{f_k}>d_\H(\ell_n^k,\gamma_m^k)$, then we have the desired inclusion \eqref{eq: end goal}. Equivalently, it suffices to show that
\begin{equation}\label{eq: end goal 2}
\tfrac{1}{2}\lvert \tr(f_k)\rvert =\cosh\tfrac{1}{2}T_{f_k}>\cosh\tfrac{1}{2}d_\H(\ell_n^k,\gamma_m^k).
\end{equation} 
For simplicity, let us write $d_0=d_\H(\ell_n^k,\gamma_m^k)$, $d_n=d_\H(\ell_n^k,\gamma_n^k)$ and $d_m=d_\H(\ell_m^k,\gamma_m^k)$. To simplify future estimates we observe that due to \eqref{eq: distance of intervals n} and \eqref{eq: distance of intervals m} we have that
\begin{align}
&\cosh \tfrac{1}{2}d_n\ \cosh \tfrac{1}{2}d_m = \cosh \tfrac{1}{2} d_\H(\ell_n^k,\gamma_n^k) \ \cosh\tfrac{1}{2}d_\H(\ell_m^k,\gamma_m^k)\nonumber\\
&=\max\left\{\sqrt{1+\left\lvert C(f_k,f_n)\right\rvert},\sqrt{1+\frac{1}{\left\lvert C(f_k,f_n)\right\rvert}}\right\} \max\left\{\sqrt{1+\left\lvert C(f_k,f_m)\right\rvert},\sqrt{1+\frac{1}{\left\lvert C(f_k,f_m)\right\rvert}}\right\} \epsilon \nonumber\\
	&\leq \left(\max_{i\neq j}\left\{ \max\left\{\sqrt{1+\left\lvert C(f_i,f_j)\right\rvert},\sqrt{1+\frac{1}{\left\lvert C(f_i,f_j)\right\rvert}}\right\} \right\}\right)^2 \epsilon \nonumber\\
	&=\max_{i\neq j} \{M_{i,j}\} \epsilon. \label{eq: proof estimate}
\end{align}
We consider two cases, in which we will use the simple inequality $\cosh(x+y)\leq 2\cosh x  \cosh y$, that holds for all $x,y\in\R$.\\
\underline{Case 1:} $d_0\leq d_n + d_m$ (see left-hand side of Figure \ref{prob})\\
Using \eqref{eq: proof epsilon} and  \eqref{eq: proof estimate} we get that
\begin{align*}
\cosh \tfrac{1}{2}d_0 &\leq \cosh \left( \tfrac{1}{2}d_n + \tfrac{1}{2}d_m\right)\leq 2\ \cosh \tfrac{1}{2}d_n \ \cosh\tfrac{1}{2} d_m\\
		&\leq 2 \ \max_{i\neq j} \{M_{i,j}\}  \cdot \epsilon	\leq 2\ \max_{i\neq j}\{M_{i,j}\}\cdot \max_{i\neq j}\left\{\frac{\lvert C(f_i,f_j)\rvert +1}{\lvert C(f_i,f_j)-1\rvert}\right\} \cdot \epsilon\\
		&< 4\ \max_{i\neq j}\{M_{i,j}\}\cdot \max_{i\neq j}\left\{\frac{\lvert C(f_i,f_j)\rvert +1}{\lvert C(f_i,f_j)-1\rvert}\right\}\cdot\epsilon < \tfrac{1}{2}\lvert \tr(f_k)\rvert ,
\end{align*}
as required for \eqref{eq: end goal 2}.\\
\underline{Case 2:} $d_0>d_n+d_m$ (see right-hand side of Figure \ref{prob})\\
In this case we have that $d_0=d_n+d_m+d_\H(\gamma_n^k,\ell_m^k)$, or equivalently 
\begin{equation}\label{eq: case 2 eq1}
\cosh \tfrac{1}{2}d_0 = \cosh\left( \tfrac{1}{2}d_n+\tfrac{1}{2}d_m + \tfrac{1}{2}d_\H(\gamma_n^k,\ell_m^k)\right) \leq 4\ \cosh \tfrac{1}{2}d_n\ \cosh\tfrac{1}{2}d_m \ \cosh \tfrac{1}{2}d_\H(\gamma_n^k,\ell_m^k).
\end{equation}
Let $H_n^k\subset \H$ be the open hyperbolic half-plane that is bounded by the geodesic $\gamma_n^k$ and contains $\ell_n^k$. Similarly, $H_m^k$ denotes the open hyperbolic half-plane bounded by $\ell_m^k$, that contains $\gamma_m^k$. Note that the half-planes $H_n^k, H_m^k$ have disjoint closures, and
\begin{equation}\label{eq: case 2 eq2}
\inf\left\{d_\H(z,w)\colon z\in H_n^k\ \text{and} \ w\in H_m^k\right\}=d_\H(\gamma_n^k,\ell_m^k).
\end{equation}
Moreover, by construction of the intervals $U_n^k, S_n^k,U_m^k, S_m^k$, we have that $\mathrm{Ax}(f_n)\subset H_n^k$, and similarly $\mathrm{Ax}(f_m)\subset H_m^k$. Therefore the axes of $f_n$ and $f_m$ are disjoint, and
\begin{equation}\label{eq: case 2 eq3}
d_\H\left( \mathrm{Ax}(f_n), \mathrm{Ax}(f_m)\right)>\inf\left\{d_\H(z,w)\colon z\in H_n^k\ \text{and} \ w\in H_m^k\right\}.
\end{equation}
So, due to \eqref{eq: case 2 eq2} and \eqref{eq: case 2 eq3}, inequality \eqref{eq: case 2 eq1} yields
\begin{align}
\cosh \tfrac{1}{2}d_0 &< 4\ \cosh \tfrac{1}{2}d_n\ \cosh\tfrac{1}{2}d_m \ \cosh\tfrac{1}{2}d_\H\left( \mathrm{Ax}(f_n), \mathrm{Ax}(f_m)\right)\nonumber \\
	&\leq 4\ \cosh \tfrac{1}{2}d_n\ \cosh\tfrac{1}{2}d_m \ \cosh d_\H\left( \mathrm{Ax}(f_n), \mathrm{Ax}(f_m)\right). \label{eq: case 2 eq4}
\end{align}
But, according to Corollary \ref{cross-coro}, we have that $\cosh d_\H\left( \mathrm{Ax}(f_n), \mathrm{Ax}(f_m)\right) = \frac{C(f_n,f_m)+1}{\lvert C(f_n,f_m)-1\rvert}$. Substituting this to \eqref{eq: case 2 eq4} and using \eqref{eq: proof epsilon}, \eqref{eq: proof estimate} yield
\begin{align*}
\cosh \tfrac{1}{2}d_0 \leq 4\ \max_{i\neq j}\{M_{i,j}\}\cdot \epsilon\cdot \frac{C(f_n,f_m)+1}{\lvert C(f_n,f_m)-1\rvert}<\tfrac{1}{2}\lvert\tr(f_k)\rvert,
\end{align*}
and we are once more led to \eqref{eq: end goal 2}.
\end{proof}

\begin{figure}[ht]
\centering
\begin{tikzpicture}[scale =3,decoration={
    markings,
    mark=at position 0.5 with {\arrow{To}}}
    ] 
    
\begin{scope}[xshift=-1.25cm]

\draw [line width=0.4pt] (0.,0.) circle (1.cm);

\draw [shift={(-1.5454824813315002,0.)},line width=0.4pt, postaction={decorate}]  plot[domain=-0.7037051677978097:0.7037051677978099,variable=\t]({1.*1.1783531304759927*cos(\t r)+0.*1.1783531304759927*sin(\t r)},{0.*1.1783531304759927*cos(\t r)+1.*1.1783531304759927*sin(\t r)});

\draw [shift={(-0.6470471274048067,-0.7650451096901801)},line width=0.4pt]  plot[domain=-0.6391519824208531:2.3766859715396405,variable=\t]({1.*0.06296034421509385*cos(\t r)+0.*0.06296034421509385*sin(\t r)},{0.*0.06296034421509385*cos(\t r)+1.*0.06296034421509385*sin(\t r)});

\draw [shift={(-0.6470471274048067,0.7650451096901801)},line width=0.4pt]  plot[domain=-2.3766859715396405:0.6391519824208534,variable=\t]({1.*0.06296034421509375*cos(\t r)+0.*0.06296034421509375*sin(\t r)},{0.*0.06296034421509375*cos(\t r)+1.*0.06296034421509375*sin(\t r)});

\draw [shift={(-0.6470471274048138,4.047607328608614)},line width=0.4pt]  plot[domain=4.624457727704823:5.1173561119792,variable=\t]({1.*3.9751471761040773*cos(\t r)+0.*3.9751471761040773*sin(\t r)},{0.*3.9751471761040773*cos(\t r)+1.*3.9751471761040773*sin(\t r)});

\draw [shift={(-0.6470471274048138,-4.047607328608585)},line width=0.4pt]  plot[domain=1.165829195200384:1.658727579474764,variable=\t]({1.*3.975147176104047*cos(\t r)+0.*3.975147176104047*sin(\t r)},{0.*3.975147176104047*cos(\t r)+1.*3.975147176104047*sin(\t r)});

\draw [shift={(-1.6316693623744682,-0.025988890184495565)},line width=0.4pt, dashed, Bar-Bar, xshift=2pt]  plot[domain=-0.035078700271639306:0.6576583832202162,variable=\t]({1.*1.2895814555601068*cos(\t r)+0.*1.2895814555601068*sin(\t r)},{0.*1.2895814555601068*cos(\t r)+1.*1.2895814555601068*sin(\t r)});

\draw [shift={(-1.6316693623744682,0.025988890184495565)},line width=0.4pt, dashed, Bar-Bar, xshift=.7pt]  plot[domain=-0.617658383220216:0.03507870027163894,variable=\t]({1.*1.2895814555601068*cos(\t r)+0.*1.2895814555601068*sin(\t r)},{0.*1.2895814555601068*cos(\t r)+1.*1.2895814555601068*sin(\t r)});
\

\draw [dashed, shift={(-1.5454824813315002,0.)},line width=0.4pt, Bar-Bar]  plot[domain=-0.645:0.645,variable=\t]({1.*1.13*cos(\t r)+0.*1.13*sin(\t r)},{0.*1.13*cos(\t r)+1.*1.13*sin(\t r)});

\node[left, xshift=-5pt] at (-0.37,0) {$f_k$}; 
\node[below left] at (-0.65, -0.76) {$S^k_m$};
\node[above right] at (-0.57, -0.87) {$\gamma_m^k$};
\node[above left] at (-0.65,0.76) {$U^k_n$};
\node[below right, yshift=1.7pt, xshift=1pt] at (-0.57,0.87) {$\ell_n^k$};
\node[right,xshift=3pt] at (-0.42,0.45) {$d_n$};
\node[right, xshift=.5pt] at (-0.43,-0.44) {$d_m$};
\node[above] at (.53,.25) {$\ell_m^k$};
\node[below] at (.53, -.25) {$\gamma_n^k$};
\node[left] at (-0.47,0.39) {$d_0$};
\end{scope}    
    
  \begin{scope}[xshift=1.25cm]

\draw [line width=0.4pt] (0.,0.) circle (1.cm);

\draw [shift={(-1.5454824813315002,0.)},line width=0.4pt, postaction={decorate}]  plot[domain=-0.7037051677978097:0.7037051677978099,variable=\t]({1.*1.1783531304759927*cos(\t r)+0.*1.1783531304759927*sin(\t r)},{0.*1.1783531304759927*cos(\t r)+1.*1.1783531304759927*sin(\t r)});

\draw [shift={(0.2362210614151159,1.0113610549987953)},line width=0.4pt, postaction={decorate}]  plot[domain=3.1855635664261706:5.780306421196157,variable=\t]({1.*0.2804488784508875*cos(\t r)+0.*0.2804488784508875*sin(\t r)},{0.*0.2804488784508875*cos(\t r)+1.*0.2804488784508875*sin(\t r)});

\draw [shift={(0.2362210614151159,-1.0113610549987953)},line width=0.4pt, postaction={decorate}]  plot[domain=0.5028788859834298:3.0976217407534157,variable=\t]({1.*0.2804488784508875*cos(\t r)+0.*0.2804488784508875*sin(\t r)},{0.*0.2804488784508875*cos(\t r)+1.*0.2804488784508875*sin(\t r)});

\draw [shift={(-0.6470471274048067,-0.7650451096901801)},line width=0.4pt]  plot[domain=-0.6391519824208531:2.3766859715396405,variable=\t]({1.*0.06296034421509385*cos(\t r)+0.*0.06296034421509385*sin(\t r)},{0.*0.06296034421509385*cos(\t r)+1.*0.06296034421509385*sin(\t r)});

\draw [shift={(-0.6470471274048067,0.7650451096901801)},line width=0.4pt]  plot[domain=-2.3766859715396405:0.6391519824208534,variable=\t]({1.*0.06296034421509375*cos(\t r)+0.*0.06296034421509375*sin(\t r)},{0.*0.06296034421509375*cos(\t r)+1.*0.06296034421509375*sin(\t r)});

\draw [shift={(-0.6470471274048138,4.047607328608614)},line width=0.4pt]  plot[domain=4.624457727704823:5.1173561119792,variable=\t]({1.*3.9751471761040773*cos(\t r)+0.*3.9751471761040773*sin(\t r)},{0.*3.9751471761040773*cos(\t r)+1.*3.9751471761040773*sin(\t r)});

\draw [shift={(-0.6470471274048138,-4.047607328608585)},line width=0.4pt]  plot[domain=1.165829195200384:1.658727579474764,variable=\t]({1.*3.975147176104047*cos(\t r)+0.*3.975147176104047*sin(\t r)},{0.*3.975147176104047*cos(\t r)+1.*3.975147176104047*sin(\t r)});

\draw [dashed, shift={(-1.6519224520504425,-0.04229476443231125)},line width=0.4pt,Bar-Bar]  plot[domain=0.1100067018665286:0.618847268440485,variable=\t]({1.*1.33*cos(\t r)+0.*1.33*sin(\t r)},{0.*1.33*cos(\t r)+1.*1.33*sin(\t r)});

\draw [dashed, shift={(-1.6519224520504425,0.04229476443231125)},line width=0.4pt, Bar-Bar]  plot[domain=5.664338038739102:6.1731786053130575,variable=\t]({1.*1.33*cos(\t r)+0.*1.33*sin(\t r)},{0.*1.33*cos(\t r)+1.*1.33*sin(\t r)});

\draw [dashed, shift={(-1.5454824813315002,0.)},line width=0.4pt, Bar-Bar]  plot[domain=-0.645:0.645,variable=\t]({1.*1.13*cos(\t r)+0.*1.13*sin(\t r)},{0.*1.13*cos(\t r)+1.*1.13*sin(\t r)});

\node[right] at (-0.37,0) {$f_k$}; 
\node[below left] at (-0.65, -0.76) {$S^k_m$};
\node[above right] at (-0.57, -0.87) {$\gamma_m^k$};
\node[above left] at (-0.65,0.76) {$U^k_n$};
\node[below right] at (-0.57,0.87) {$\ell_n^k$};
\node[right] at (-0.42,0.45) {$d_n$};
\node[right] at (-0.43,-0.44) {$d_m$};
\node[below] at (.19,.74) {$f_n$};
\node[above] at (.19,-.74) {$f_m$};
\node[above] at (.53,.25) {$\gamma_n^k$};
\node[below] at (.53, -.25) {$\ell_m^k$};
\node[left] at (-0.47,0.39) {$d_0$};
\end{scope}

\end{tikzpicture}
\caption{The two cases in the proof of Theorem \ref{MAIN}.}\label{prob}
\end{figure}

\section*{Acknowledgments}
The research project is implemented in the framework of H.F.R.I call ``3rd Call for H.F.R.I.'s Research Projects to Support Faculty Members \& Researchers” (H.F.R.I. Project Number: 24979).


\begin{bibdiv}
\begin{biblist}


\bib{AvBoYo2010}{article}{
   author={Avila, Artur},
   author={Bochi, Jairo},
   author={Yoccoz, Jean-Christophe},
   title={Uniformly hyperbolic finite-valued ${\rm SL}(2,\Bbb R)$-cocycles},
   journal={Comment. Math. Helv.},
   volume={85},
   date={2010},
   number={4},
   pages={813--884},
}


\bib{BaKaMo2020}{article}{
   author={B\'ar\'any, Bal\'azs},
   author={K\"aenm\"aki, Antti},
   author={Morris, Ian D.},
   title={Domination, almost additivity, and thermodynamic formalism for
   planar matrix cocycles},
   journal={Israel J. Math.},
   volume={239},
   date={2020},
   number={1},
   pages={173--214}
}

\bib{BBC1996}{article}{
   author={B\'ar\'any, I.},
   author={Beardon, A. F.},
   author={Carne, T. K.},
   title={Barycentric subdivision of triangles and semigroups of M\"obius
   maps},
   journal={Mathematika},
   volume={43},
   date={1996},
   number={1},
   pages={165--171}
}

\bib{BaVi2012}{article}{
   author={Barnsley, Michael F.},
   author={Vince, Andrew},
   title={Real projective iterated function systems},
   journal={J. Geom. Anal.},
   volume={22},
   date={2012},
   number={4},
   pages={1137--1172}
}

\bib{Be1995}{book}{
   author={Beardon, Alan F.},
   title={The geometry of discrete groups},
   series={Graduate Texts in Mathematics},
   volume={91},
   note={Corrected reprint of the 1983 original},
   publisher={Springer-Verlag, New York},
   date={1995},
   pages={xii+337},
}

\bib{BlMo2019}{article}{
   author={Blumenthal, Alex},
   author={Morris, Ian D.},
   title={Characterization of dominated splittings for operator cocycles
   acting on Banach spaces},
   journal={J. Differential Equations},
   volume={267},
   date={2019},
   number={7},
   pages={3977--4013}
}



\bib{BoGo2009}{article}{
   author={Bochi, Jairo},
   author={Gourmelon, Nicolas},
   title={Some characterizations of domination},
   journal={Math. Z.},
   volume={263},
   date={2009},
   number={1},
   pages={221--231}
   }
   
   
\bib{BoMo2015}{article}{
   author={Bochi, Jairo},
   author={Morris, Ian D.},
   title={Continuity properties of the lower spectral radius},
   journal={Proc. Lond. Math. Soc. (3)},
   volume={110},
   date={2015},
   number={2},
   pages={477--509}
}

\bib{BoPoSa2019}{article}{
   author={Bochi, Jairo},
   author={Potrie, Rafael},
   author={Sambarino, Andr\'es},
   title={Anosov representations and dominated splittings},
   journal={J. Eur. Math. Soc. (JEMS)},
   volume={21},
   date={2019},
   number={11},
   pages={3343--3414}
}


\bib{BoVi2005}{article}{
   author={Bochi, Jairo},
   author={Viana, Marcelo},
   title={The Lyapunov exponents of generic volume-preserving and symplectic
   maps},
   journal={Ann. of Math. (2)},
   volume={161},
   date={2005},
   number={3},
   pages={1423--1485}
}


\bib{BoDiPu2003}{article}{
   author={Bonatti, C.},
   author={D\'iaz, L. J.},
   author={Pujals, E. R.},
   title={A $C^1$-generic dichotomy for diffeomorphisms: weak forms of
   hyperbolicity or infinitely many sinks or sources},
   language={English, with English and French summaries},
   journal={Ann. of Math. (2)},
   volume={158},
   date={2003},
   number={2},
   pages={355--418}
}

\bib{Ch2022}{article}{
   author={Christodoulou, Argyrios},
   title={Parameter spaces of locally constant cocycles},
   journal={Int. Math. Res. Not. IMRN},
   date={2022},
   number={17},
   pages={13590--13628}
}

\bib{ChJu2021}{article}{
   author={Christodoulou, Argyrios},
   author={Jurga, Natalia},
   title={The Hausdorff dimension of self-projective sets},
   journal={Trans. Amer. Math. Soc.},
   volume={379},
   date={2026},
   number={1},
   pages={1--32}
}


\bib{CoKl2000}{book}{
   author={Colonius, Fritz},
   author={Kliemann, Wolfgang},
   title={The dynamics of control},
   series={Systems \& Control: Foundations \& Applications},
   note={With an appendix by Lars Gr\"une},
   publisher={Birkh\"auser Boston, Inc., Boston, MA},
   date={2000},
   pages={xii+629}
}

\bib{DaFiGo2022}{article}{
   author={Damanik, David},
   author={Fillman, Jake},
   author={Gohlke, Philipp},
   title={Spectral characteristics of Schr\"odinger operators generated by
   product systems},
   journal={J. Spectr. Theory},
   volume={12},
   date={2022},
   number={4},
   pages={1659--1718}
}

\bib{He2015}{article}{
   author={Herrlich, Frank},
   title={Schottky space and Teichm\"uller disks},
   conference={
      title={Handbook of group actions. Vol. I},
   },
   book={
      series={Adv. Lect. Math. (ALM)},
      volume={31},
      publisher={Int. Press, Somerville, MA},
   },
   isbn={978-1-57146-300-5},
   date={2015},
   pages={289--308}
}


\bib{JS}{article}{
   author={Jacques, Matthew},
   author={Short, Ian},
   title={Semigroups of isometries of the hyperbolic plane},
   journal={Int. Math. Res. Not. IMRN},
   date={2022},
   number={9},
   pages={6403--6463}
}

\bib{Jo1976}{article}{
   author={J\o rgensen, Troels},
   title={On discrete groups of M\"{o}bius transformations},
   journal={Amer. J. Math.},
   volume={98},
   date={1976},
   number={3},
   pages={739--749}
}



\bib{Ma1978}{article}{
   author={Ma\~n\'e, Ricardo},
   title={Contributions to the stability conjecture},
   journal={Topology},
   volume={17},
   date={1978},
   number={4},
   pages={383--396}
}

\bib{Ma1984}{article}{
   author={Ma\~n\'e, Ricardo},
   title={Oseledec's theorem from the generic viewpoint},
   conference={
      title={Proceedings of the International Congress of Mathematicians,
      Vol.\ 1, 2},
      address={Warsaw},
      date={1983},
   },
   book={
      publisher={PWN, Warsaw},
   },
   isbn={83-01-05523-5},
   date={1984},
   pages={1269--1276}
}

\bib{Ma1967}{article}{
   author={Maskit, Bernard},
   title={A characterization of Schottky groups},
   journal={J. Analyse Math.},
   volume={19},
   date={1967},
   pages={227--230}
}

\bib{Ma1998}{article}{
   author={Maskit, Bernard},
   title={On spaces of classical Schottky groups},
   conference={
      title={In the tradition of Ahlfors and Bers},
      address={Stony Brook, NY},
      date={1998},
   },
   book={
      series={Contemp. Math.},
      volume={256},
      publisher={Amer. Math. Soc., Providence, RI},
   },
   isbn={0-8218-1371-4},
   date={2000},
   pages={227--237}
}


\bib{QTZ2019}{article}{
   author={Quas, Anthony},
   author={Thieullen, Philippe},
   author={Zarrabi, Mohamed},
   title={Explicit bounds for separation between Oseledets subspaces},
   journal={Dyn. Syst.},
   volume={34},
   date={2019},
   number={3},
   pages={517--560}
}


\bib{Sa2016}{article}{
   author={Sambarino, Mart\'in},
   title={A (short) survey on dominated splittings},
   conference={
      title={Mathematical Congress of the Americas},
   },
   book={
      series={Contemp. Math.},
      volume={656},
      publisher={Amer. Math. Soc., Providence, RI},
   },
   isbn={978-1-4704-2310-0},
   date={2016},
   pages={149--183}
}

\bib{SoTa2021}{article}{
   author={Solomyak, Boris},
   author={Takahashi, Yuki},
   title={Diophantine property of matrices and attractors of projective
   iterated function systems in $\Bbb{R}\Bbb{P}^1$},
   journal={Int. Math. Res. Not. IMRN},
   date={2021},
   number={16},
   pages={12639--12669}
}

\bib{Wo2022}{article}{
   author={Wood, William},
   title={On the spectrum of the periodic Anderson-Bernoulli model},
   journal={J. Math. Phys.},
   volume={63},
   date={2022},
   number={10},
   pages={Paper No. 102705, 16}
}

\bib{Yo2004}{article}{
   author={Yoccoz, Jean-Christophe},
   title={Some questions and remarks about ${\rm SL}(2,\bold R)$ cocycles},
   conference={
      title={Modern dynamical systems and applications},
   },
   book={
      publisher={Cambridge Univ. Press, Cambridge},
   },
   date={2004},
   pages={447--458},
}


\end{biblist}
\end{bibdiv}
\end{document}